\documentclass[a4paper,twoside]{article}
\usepackage{amsthm,amscd}
\usepackage{amsmath,amssymb,amsfonts}
\usepackage[T1]{fontenc}
\usepackage{tikz}
\usetikzlibrary{trees}
\usepackage{pdfpages}

\usepackage{caption}
\usepackage[utf8]{inputenc}
\usepackage[english]{babel}
\usepackage{graphicx}               
\usepackage{color}                  
\RequirePackage[colorlinks=true,citecolor=blue,linkcolor=blue]{hyperref}
\usepackage{hyperref}
\usepackage[absolute]{textpos} 
\usepackage{multicol}
\usepackage{array}

\usepackage[babel=true]{csquotes} 

\usepackage{bbm}
\newcommand{\ind}{\mathbbm{1}}

\newtheorem{theorem}{Theorem}[section]

\newtheorem{prop}[theorem]{Proposition}

\theoremstyle{definition}
\newtheorem{definition}[theorem]{Definition}
\newtheorem{ex}[theorem]{Example}

\def \leq{\leqslant}

\newcommand{\abs}[1]{\left\lvert #1 \right\rvert}

\newcommand{\eqdef}{\mathrel{=}:}

\newcommand{\bmu}{\mbox{$\boldsymbol \mu$}}

\usepackage{authblk}
 
\newcommand{\beq} {\begin{eqnarray*}}
\newcommand{\eeq} {\end{eqnarray*}}

\def \R{\mathbb{R}}
\def \Q{\mathbb{Q}}
\def \N{\mathbb{N}}

\def \ZZ{\mathbf{Z}}

\def \E{\mathbb{E}}
\def \F{\mathbb{F}}
\def \P{\mathbb{P}}
\def \Var{\hbox{{\rm Var}}}

\def \Cov{\hbox{{\rm Cov}}}

\DeclareMathOperator*{\argmin}{Argmin}

\newcommand{\1}{{1\hspace{-0.2ex}\rule{0.12ex}{1.61ex}\hspace{0.5ex}}}

\newcommand{\z}{\mathbf{z}}

\title{Global sensitivity analysis and Wasserstein spaces} 

\author[1]{Jean-Claude Fort}
\author[2]{Thierry Klein} 
\author[3]{Agn\`es Lagnoux}
\affil[1]{MAP5 Universit\'e Paris Descartes, SPC, 45 rue des Saints P\`eres, 75006 Paris, France.}
\affil[2]{Institut de Math\'ematiques de Toulouse; UMR5219. Universit\'e de Toulouse; ENAC - Ecole Nationale de l'Aviation Civile , Universit\'e de Toulouse, France}
\affil[3]{Institut de Math\'ematiques de Toulouse; UMR5219. Universit\'e de Toulouse; CNRS. UT2J, F-31058 Toulouse, France.}

\begin{document}

\maketitle

\begin{abstract}
Sensitivity indices are commonly used to quantity the relative influence of any specific group of input variables on the output of a computer code. In this paper, we focus both on computer codes the output of which is a cumulative distribution function  and on stochastic computer codes. We propose a way to perform a global sensitivity analysis for these kinds of computer codes. In the first setting, we define two indices: the first one is based on Wasserstein Fr\'echet means while the second one is based on the Hoeffding decomposition of the indicators of Wasserstein balls. Further, when dealing with the stochastic computer codes, we define an ``ideal version'' of the stochastic computer code thats fits into the frame of the first setting. Finally, we deduce a procedure to realize a second level global sensitivity analysis, namely when one is interested in the sensitivity related to the input distributions rather than in the sensitivity related to the inputs themselves. Several numerical studies are proposed as illustrations in the different settings. 
\end{abstract}


\medskip



\section{Introduction}

The use of complex computer models for the analysis of applications from  sciences, engineering and other fields is by now routine. For instance, in the area of marine submersion, complex computer codes have been  developed  to simulate submersion events (see, e.g.,  \cite{betancourt:hal-01998724,idier:hal-02458084} for more details). In the context of aircraft design, sensitivity analysis and metamodelling  are intensively used to optimize the design of an airplane (see, e.g.,  \cite{peteilh:hal-02866381}). Several other concrete examples of stochastic computer codes can be found in \cite{marrel2012global}.

Often, the models are expensive to run in terms of computational time. Thus it is crucial to understand the global influence of  one or several inputs  on  the  output of the system under study  with a moderate number of runs afforded \cite{sant:will:notz:2003}. When these inputs are regarded as random elements, this   problem  is  generally  called  (global)
sensitivity analysis.
  We refer to \cite{rocquigny2008uncertainty,saltelli-sensitivity,sobol1993}
 for an overview of the practical aspects of global sensitivity analysis.

A classical tool to perform global sensitivity analysis  consists in computing the Sobol indices. These indices were first introduced in \cite{pearson1915partial} and then considered by \cite{sobol2001global}. They are well tailored when the output space is $\R$. The Sobol indices compare, using the Hoeffding decomposition \cite{Hoeffding48}, the conditional variance of the output knowing some of the input variables to the total variance of the output. Many different estimation procedures of the Sobol indices have been proposed and studied in the literature. Some are based on Monte-Carlo or quasi Monte-Carlo design of experiments (see \cite{Kucherenko2017different,owen2} and references therein for more details). More recently a method based on nested Monte-Carlo \cite{GODA201763} has been developed. 
In particular, an efficient estimation of the Sobol indices can be performed through the so-called Pick-Freeze method. For the description of this method and its theoretical study (consistency, central limit theorem, concentration inequalities and Berry- Esseen bounds), we refer to \cite{janon2012asymptotic,pickfreeze} and references therein. Some other estimation procedures are based on different designs of experiments using for example polynomial chaos expansions (see \cite{Sudret2008global} and the reference therein for more details).

Since Sobol indices are variance based, they only quantify the influence of the inputs on the mean behaviour of the code. Many authors proposed other criteria to compare the conditional distribution of the output knowing some of the inputs to the distribution of the output. In \cite{owen2,ODC13,Owen12}, the authors use higher moments to define new indices while, in \cite{borgonovo2007,borgonovo2011moment, DaVeiga13}, the use of divergences or distances between measures allows to define new indices. In \cite{FKR13}, the authors use contrast functions to build indices that are goal oriented. Although these works define nice theoretical indices, the existence of a relevant statistical estimation procedure is still in most cases an open question. The case of vectorial-valued computer codes is considered in \cite{GKL18} where a sensitivity index based on the whole distribution utilizing the Cram\'er-von-Mises distance is defined. Within this framework, the authors show that the Pick-Freeze estimation procedure provides an asymptotically Gaussian estimator of the index. The definition of the Cramér-von-Mises indices has been extended to computer codes valued in general metric spaces in \cite{FGM2017,GKLM19}.

Nowadays, the computer code output is often no longer a real-valued multidimensional variable but rather a function computed at various locations. In that sense, it can be considered as a functional output. Some other times, the computer code is stochastic in the sense  that the same inputs can lead to different outputs.
When the output of the computer code is a function (for instance, a cumulative distribution function) or when the computer code is stochastic, Sobol indices are no longer well tailored. It is then crucial to define indices adapted to the functional or random aspect of the output. When the output is vectorial or valued in an Hilbert space some  generalizations  of Sobol indices are available \cite{lamboni2011multivariate,GJKL14}.  Nevertheless, these indices are still based on the Hoeffding decomposition of the output; so that they only quantify the relative influence of an input through the variance. More recently, indices based on the whole distribution have been developed
\cite{DaVeiga13,BI15,borgonovo2007}. In particular, the method relying on Cram\'er-von-Mises distance \cite{GKL18} compares the conditionnal cumulative distribution function with the unconditional  one by considering the Hoeffding  decomposition of half-space indicators (rather than the Hoeffding  decomposition of the output itself) and by integrating them. This method was then extend to codes taking values in a Riemannian manifold \cite{FGM2017} and then in general metric spaces  \cite{GKLM19}.

In this work, we focus on two kinds of computer codes: 1) computer codes the  output of which is  the cumulative distribution function of a real random variable and 2) real-valued stochastic computer codes. A first step will consist in performing global sensitivity analysis for these kinds of computer codes. Further, we will deduce how to perform second level sensitivity analysis using the tools developed in the first step. 
A code with cumulative distribution function  as output can be seen as a code taking values in the space of all probability measures on $\R$. This space can be endowed with a metric (for example, the Wasserstein metric \cite{villani2003tot}). This point of view allows to define at least two different indices for this kind of codes, generalizing the framework of  \cite{GKLM19}. The first one is based on Wasserstein Fr\'echet means while the second one is based on the Hoeffding decomposition of the indicators of Wasserstein balls. 
Further, stochastic codes  (see  Section \ref{sec:codesto} for a bibliographical study) can be seen as a ``discrete approximation'' of codes having cumulative distribution functions as values. Then it is possible to define ``natural'' indices for such stochastic codes. Finally, second level sensitivity analysis aims at considering uncertainties on the type of the input distributions and/or on the parameters of the input distributions (see  Section \ref{sec:input} for a bibliographical study). Actually, this kind of problem can be embedded in the framework of stochastic codes.

The article is organized as follows.
In Section  \ref{SIcodes}, we introduce and precisely define a general class of global sensitivity indices. We also present statistical methods to estimate these indices. In Section \ref{sec:wass}, we recall some basic facts on Wasserstein distances, Wasserstein costs and Fr\'echet means. In Section \ref{sec:GSAWS}, we define and study the statistical properties of  two new global sensitivity indices for computer codes valued in general Wasserstein spaces. Further, in Section \ref{sec:codesto}, we study the case of stochastic computer codes. Finally, Section \ref{sec:input} is dedicated to the sensitivity analysis with respect to the distributions of the input variables.

\section{Sensitivity indices for codes valued in general metric spaces}\label{SIcodes}

We consider a black-box code $f$ defined on  a product of measurable spaces $E=E_1\times E_2\times\ldots\times E_p$ ($p\in \N^*$) taking its values in a metric space $\mathcal{X}$. The output denoted by $Z$ is then given by
\begin{equation}\label{eq:model}
Z=f(X_1,\ldots,X_p).
\end{equation}
We denote by $\P$ the distribution of the output code $Z$.

\medskip

The aim of this work is to give some partial answers to the following questions.
\begin{enumerate}
\item [\bf{Question 1}] How can we perform Global Sensitivity Analysis (GSA) when the output space is the space of probability distribution functions (p.d.f.) on $\R$ or  the space of cumulative distribution functions (c.d.f.)?
\item [\bf{Question 2}] How can we perform GSA for stochastic computer codes?
\item [\bf{Question 3}] How can we perform GSA with respect to the choice of the distributions of the input variables?
\end{enumerate}

\subsection{The general metric spaces sensitivity index}\label{ssec:GMSindex}

In \cite{GKLM19}, the authors performed GSA for codes $f$ taking values in  general metric spaces.  To do so, they consider a family of test functions  parameterized by $m
 \in \N^*$ elements of $\mathcal X$ and defined by 
 \[
\begin{matrix}
& \mathcal X^{m}\times \mathcal X & \to & \R\\
& (a,x) & \mapsto & T_a(x).\\
\end{matrix}
\]
Let $\textbf{u}\subset \{1,\ldots,p\}$ and $X_{\textbf{u}}=(X_i,i\in \textbf{u})$. 
 Assuming that the test functions $T_a$ are  L$^2$-functions with respect to the product measure $\P^{\otimes m}\otimes \P$ (where $\P^{\otimes m}$ is the product $m$-times of 
the distribution of the output code $Z$) on $\mathcal X^{m}\times \mathcal X$,
they allow to defined the general metric space sensitivity index  with respect to $X_{\textbf{u}}$  by
\begin{align}\label{eq:GIM}
S_{2,\text{GMS}}^{\textbf{u}}=\frac{\int_{\mathcal X^m}\E\left[\left(\E[T_a(Z)]-\E[T_a(Z)|X_{\textbf{u}}]\right)^{2}\right]d\P^{\otimes m}(a)}{\int_{\mathcal X^m} \Var(T_a(Z))d\P^{\otimes m}(a)}
=\frac{\int_{\mathcal X^m}\Var\left(\E[T_a(Z)|X_{\textbf{u}}]\right)d\P^{\otimes m}(a)}{\int_{\mathcal X^m} \Var(T_a(Z))d\P^{\otimes m}(a)}.
\end{align}
Roughly speaking, there are two parts in the previous indices. First, for any value of $a$, we consider the numerator $\E\bigl[\left(\E[T_a(Z)]-\E[T_a(Z)|X_{\textbf{u}}]\right)^{2}\bigr]$ and the denominator $\Var(T_a(Z))$ of the classical Sobol index of $T_a(Z)$. We call this part the Sobol' part. Second, we integrate each part with respect to the measure  $\P^{\otimes m}$; this will be called the integration part.

\medskip

As explained in \cite{GKLM19}, by construction, the indices $S_{2,\text{GMS}}^{\textbf{u}}$ lie in $[0,1]$ and share the same properties as their Sobol counterparts:
\begin{enumerate}
\item the different contributions sum to 1; 
\item they are invariant by translation, by any isometric and by any non-degenerated scaling of $Z$. 
\end{enumerate}

\paragraph{Estimation}

Three different estimation procedures are available in this context. The two first methods are based on the so-called Pick-Freeze scheme.
More precisely, the Pick-Freeze scheme, considered in  \cite{janon2012asymptotic}, is 
a well tailored design of experiment. Namely, let $X^{\textbf{u}}$ be the random  vector such
that $X^{\textbf{u}}_{i}=X_{i}$ if $i\in {\textbf{u}}$ and $X^{\textbf{u}}_i=X'_i$ if $i\notin {\textbf{u}}$ where $X'_i$ is an independent copy of $X_i$. We then set 
\begin{align}\label{def:Yv}
Z^{{\textbf{u}}}:=f(X^{\textbf{u}}).
\end{align}
Further, the procedure consists in rewriting the variances of the conditional expectation in terms of covariances as follows:
\begin{align}\label{eq:trick_PF}
\Var(\E[Z|X^{\textbf u}])=\Cov(Z,Z^{\textbf u}).
\end{align}
Alternatively, the third estimation procedure that can be seen as an ingenious and effective approximation of the Pick-Freeze scheme  is based on rank statistics. Until now, it is unfortunately only available to estimate first order indices in the case of real-valued inputs.

\begin{itemize}
\item \underline{First method - Pick-Freeze}. Introduced in \cite{GKL18}, this procedure   is based on a double Monte-Carlo scheme to estimate the Cram\'er-von-Mises indices $S^{\textbf{u}}_{2,\text{CVM}}$. 
More precisely, to estimate $S_{2,\text{GMS}}^{\textbf{u}}$ in our context, one  considers the following design of experiment consisting in:
\begin{enumerate}
\item a classical Pick-Freeze $N$-sample, that is two $N$-samples of $Z$: $(Z_j,Z_j^{\textbf{u}})$, $1\leqslant j\leqslant N$;
\item $m$ another $N$-samples of $Z$ independent of $(Z_j,Z_j^{\textbf{u}})_{1\leqslant j\leqslant N}$: $W_{l,k}$, $1\leqslant l\leqslant m$, $1\leqslant k\leqslant N$.
\end{enumerate}

The empirical estimator of the numerator of $S_{2,\text{GMS}}^{\textbf{u}}$ is then given by
\begin{align*}
\widehat N_{2,\text{GMS},\text{PF}}^{\textbf{u}}=&\frac{1}{N^m} \sum_{1\leq i_1,\dots,i_{m}\leq N} \biggl[\frac 1N \sum_{j=1}^N  T_{W_{1,i_1},\cdots,W_{m,i_m}}(Z_j)T_{W_{1,i_1},\cdots,W_{m,i_m}}(Z_j^{\textbf u}) \biggr]
\\
&-\frac{1}{N^m} \sum_{1\leq i_1,\dots,i_{m}\leq N} \biggl[\frac{1}{2N} \sum_{j=1}^N  \left(T_{W_{1,i_1},\cdots,W_{m,i_m}}(Z_j)+T_{W_{1,i_1},\cdots,W_{m,i_m}}(Z_j^{\textbf u}) \right)\biggr]^2
\end{align*}
while the one of the denominator is 
\begin{align*}
\widehat D_{2,\text{GMS},\text{PF}}^{\textbf{u}}=&\frac{1}{N^m} \sum_{1\leq i_1,\dots,i_{m}\leq N} \biggl[\frac {1}{2N} \sum_{j=1}^N  \left(T_{W_{1,i_1},\cdots,W_{m,i_m}}(Z_j)^2+T_{W_{1,i_1},\cdots,W_{m,i_m}}(Z_j^{\textbf u})^2\right) \biggr]
\\
&-\frac{1}{N^m} \sum_{1\leq i_1,\dots,i_{m}\leq N} \biggl[\frac{1}{2N} \sum_{j=1}^N  \left(T_{W_{1,i_1},\cdots,W_{m,i_m}}(Z_j)+T_{W_{1,i_1},\cdots,W_{m,i_m}}(Z_j^{\textbf u}) \right)\biggr]^2.
\end{align*}

For $\mathcal X=\R^k$, $m=1$, and $T_a$ given by $T_a(x)=\ind_{\{x\leqslant a\}}$, the index $S_{2,\text{GMS},\text{PF}}^{\textbf{u}}$ is nothing more than the index $S_{2,\text{CVM}}^{\textbf{u}}$ defined in \cite{GKL18} based on the Cram\'er-von-Mises distance and the whole distribution of the output. 
Its estimator $\widehat S_{2,\text{CVM}}^{\textbf{u}}$ defined as the ratio of $\widehat N_{2,\text{GMS},\text{PF}}^{\textbf{u}}$ and $\widehat D_{2,\text{GMS},\text{PF}}^{\textbf{u}}$ with $T_a(x)=\ind_{\{x\leqslant a\}}$ has been proved to be  asymptotically Gaussian \cite[Theorem 3.8]{GKL18}. The proof relies on Donsker's theorem and the functional delta method \cite[Theorem 20.8]{van2000asymptotic}. Hence, in the general case of $S_{2,\text{GMS}}^{\textbf{u}}$, 
the central limit theorem will be still valid as soon as the collection $(T_a)_{a\in \mathcal{X}^m}$ forms a Donsker's class of functions.

\item \underline{Second method - U-statistics}. As done in \cite{GKLM19}, this method allows the practitioner to get rid of the additional random variables $(W_{l,k})$ for $l\in\{1,\ldots,m\}$ and $k\in\{1,\ldots,N\}$. The estimator is now based on U-statistics and deals simultaneously with the Sobol part and  the integration part with respect to $d\P^{\otimes m}(a)$. It suffices to rewrite $S_{2,\text{GMS}}^{\textbf{u}}$ as 
\begin{equation}\label{def:Ustat}
S_{2,\text{GMS}}^{\textbf{u}}= \frac{I(\Phi_1)-I(\Phi_2)}{I(\Phi_3)-I(\Phi_4)},
\end{equation}
where,
\begin{align}\label{def:Phi}
&\Phi_1(\z_1,\dots,\z_{m+1})= T_{z_1,\dots,z_m}(z_{m+1})T_{z_1,\dots,z_m}(z_{m+1}^{\textbf{u}}),\nonumber\\
&\Phi_2(\z_1,\dots,\z_{m+2})= T_{z_1,\dots,z_m}(z_{m+1})T_{z_1,\dots,z_m}(z_{m+2}^{\textbf{u}}),\\
&\Phi_3(\z_1,\dots,\z_{m+1})= T_{z_1,\dots,z_m}(z_{m+1})^2,\nonumber\\
&\Phi_4(\z_1,\dots,\z_{m+2})= T_{z_1,\dots,z_m}(z_{m+1})T_{z_1,\dots,z_m}(z_{m+2}),\nonumber
\end{align}
denoting by $\z_i$ the pair $(z_i,z_i^{\textbf{u}})$
and, for $l=1,\dots,4$,
\begin{align}\label{def:I}
&I(\Phi_l)=\int_{\mathcal X^{m(l)}}\Phi_l(\z_1,\dots,\z_{m(l)})d\P_2^{u,\otimes m(l)}(\z_1\dots,\z_{m(l)}),
\end{align}
with $m(1)=m(3)=m+1$ and $m(2)=m(4)=m+2$. 
Finally, one considers the empirical version of \eqref{def:Ustat} as estimator of ${S}_{2,\text{GMS}}^{\textbf{u}}$:
\begin{equation}\label{def:estUstat}
\widehat{S}_{2,\text{GMS},\text{Ustat}}^{\textbf{u}}= \frac{U_{1,N}-U_{2,N}}{U_{3,N}-U_{4,N}},
\end{equation}
where, for $l=1,\dots,4$, 
\begin{align}\label{def:estU}
U_{l,N}&= \begin{pmatrix}N\\m(l)\end{pmatrix}^{-1}\sum_{1\leq i_1<\dots<i_{m(l)}\leq N}\Phi_l^s\left(
\ZZ_{i_1},\dots,\ZZ_{i_{m(l)}}
\right)
\end{align}
and the function: 
\begin{align*}
\Phi_l^s(\z_1,\dots,\z_{m(l)})=\frac{1}{(m(l))!} \sum_{\tau \in \mathcal S_{m(l)}} \Phi_l(\z_{\tau(1)},\dots,\z_{\tau(m(l))})
\end{align*}
is the symmetrized version of $\Phi_l$. 
 In \cite[Theorem 2.3]{GKLM19}, the estimator $\widehat{S}_{2,\text{GMS}}^{\textbf{u}}$ has been proved to be consistent and asymptotically Gaussian.

Even if the Pick-Freeze procedure is quite general, it presents some drawbacks. First of all, the Pick-Freeze design of experiment is peculiar and may not be available in real applications. Moreover, it can be unfortunately very time consuming in practice. For instance, estimating all the first order Sobol indices requires $(p+1)N$ calls to the computer code.

\item \underline{Third method - Rank-based}. In \cite{Chatterjee2019}, Chatterjee proposes an efficient way based on ranks to estimate a new coefficient of correlation. 
This estimation procedure can be seen as an approximation of the Pick-Freeze scheme and then has been exploited in \cite{GGKL20} to perform a more efficient estimation of $S_{2,\text{GMS}}^{\textbf{u}}$. Anyway, this method is only well tailored for estimating first order indices  i.e.\ in  the case of $\textbf{u}=\{i\}$ for some $i\in \{1,\ldots,p\}$ and when the input $X_i\in\R$.

More precisely, an i.i.d.\ sample of pairs of real-valued random variables $(X_{i,j},Y_j)_{1\leqslant j\leqslant N}$ ($i\in\{1,\cdots,p\}$) is considered, assuming for simplicity that the  laws of $X_i$ and $Y$ are both diffuse (ties are excluded). The pairs $(X_{i,(1)},Y_{(1)}),\ldots,(X_{i,(N)},Y_{(N)})$ are rearranged
in such a way that 
\[
X_{i,(1)}< \ldots< X_{i,(N)}
\]
and, for any $j=1,\ldots,N$, $Y_{(j)}$ is the output computed from $X_{i,(j)}$. 
Let $r_j$ be the rank of $Y_{(j)}$, that is, 
\[
r_j=\#\{j'\in \{1,\dots,N\},\ Y_{(j')}\leqslant Y_{(j)}\}.
\]
The new correlation coefficient  is then given by
\begin{equation}\label{eq:coeffChatterjee}
\xi_N(X_i,Y) = 1-\frac{3\sum_{j=1}^{N-1}|r_{j+1}-r_j|}{N^2-1}.
\end{equation}

In \cite{Chatterjee2019}, it is proved that $\xi_N(X,Y)$ converges almost surely to a deterministic limit $\xi(X,Y)$ which is actually equal to $S_{2,\text{CVM}}^i$ when $Y=Z=f(X_1,\cdots,X_p)$. Further, the author also proves a central limit theorem when $X_i$ and $Y$ are independent, which is clearly not relevant in the context of sensitivity analysis (where $X_i$ and $Y$ are assumed to be dependent through the computer code).

In our context, recall that $\textbf{u}=\{i\}$ and let $Y=Z$. Let also $\pi_i(j)$ be the rank of $X_{i,j}$ in the sample $(X_{i,1},\ldots,X_{i,N})$ of $X_i$ and  define
\begin{align}\label{def:N}
N_i(j)=\begin{cases}
\pi_i^{-1} (\pi_i(j)+1) & \text{if  $\pi_i(j)+1\leqslant N$},\\
\pi_i^{-1} (1) & \text{if  $\pi_i(j)=N$}.\\
\end{cases}
\end{align}

Then the empirical estimator $\widehat S_{2,\text{GMS},\text{Rank}}^{i}$ of $S_{2,\text{GMS}}^{i}$ only requires a $N$-sample $(Z_j)_{1\leqslant j\leqslant N}$ of $Z$ and is given by the ratio between
\begin{align*}
\widehat N_{2,\text{GMS},\text{Rank}}^{i}=&\frac{1}{N^m} \sum_{1\leq i_1,\dots,i_{m}\leq N} \biggl[\frac 1N \sum_{j=1}^N  T_{ Z_{i_1},\cdots, Z_{i_m}}(Z_j)T_{ Z_{i_1},\cdots, Z_{i_m}}(Z_{N_i(j)}) \biggr]
\\
&-\frac{1}{N^m} \sum_{1\leq i_1,\dots,i_{m}\leq N} \biggl[\frac{1}{N} \sum_{j=1}^N  T_{ Z_{i_1},\cdots, Z_{i_m}}(Z_j)\biggr]^2
\end{align*}
and $\widehat D_{2,\text{GMS},\text{Rank}}$ 
\begin{align*}
\frac{1}{N^m} \sum_{1\leq i_1,\dots,i_{m}\leq N} \biggl[\frac {1}{N} \sum_{j=1}^N  T_{ Z_{i_1},\cdots, Z_{i_m}}(Z_j)^2 \biggr]
-\frac{1}{N^m} \sum_{1\leq i_1,\dots,i_{m}\leq N} \biggl[\frac{1}{N} \sum_{j=1}^N  T_{ Z_{i_1},\cdots, Z_{i_m}}(Z_j)\biggr]^2.
\end{align*}

It is worth mentioning that $Z_{N_i(j)}$ plays the same role as $Z^{i}_j$ (the Pick-Freeze version of $Z_j$) in the Pick-Freeze estimation procedure. 
Anyway, the strength of the rank-based estimation procedure lies in the fact that only one $N$-sample of $Z$ is required while $(m+2)$ samples of size $N$ are necessary in the Pick-Freeze estimation of a single index (worse, $(m+1+p)$ samples of size $N$ are required when one wants to estimates $p$ indices).

\end{itemize}

\paragraph{Comparison of the estimation procedures}

\indent

First, the Pick-Freeze estimation procedure allows the estimation of  several sensitivity indices: the classical Sobol indices for real-valued outputs, as well as their generalization for vectorial-valued codes, but also the indices based on higher moments  \cite{ODC13} and 
the Cram\'er-von-Mises indices which take into account on the whole distribution \cite{GKL18,FGM2017}. Moreover, the Pick-Freeze estimators have desirable statistical properties. More precisely, this estimation scheme has  been proved to be consistent and asymptotically normal (i.e.\ the rate of convergence is $\sqrt{N}$) in \cite{janon2012asymptotic,pickfreeze,GKLM19}. The limiting variances can be computed explicitly, allowing the practitioner to build confidence intervals.  In addition, for a given sample size $N$, exponential inequalities have been established. 
Last but not least, the sequence of estimators is asymptotically efficient from such a design of experiment (see, \cite{van2000asymptotic} for the definition of the asymptotic efficiency and \cite{pickfreeze} for more details on the result). 

However, the Pick-Freeze estimators  have two major  drawbacks. 
First, they rely on a particular experimental design that may be unavailable in practice. Second, the number of model calls to estimate all first order Sobol indices
grows linearly with the number of input parameters.  For example, if we consider $p=99$ input parameters  and only  $n=1000$ calls are allowed, then only a sample of size $n/(p+1)=10$ is available to estimate each single first order Sobol index.

\medskip

Secondly, the estimation procedure based on U-statistics has the same kind of asymptotic guarantees as consistency and asymptotic normality.  
Furthermore, the estimation scheme is reduced to $2N$  evaluations of the code. Last but not least, using the results of Hoeffding \cite{Hoeffding48} on U-statistics, the asymptotic normality is proved straightforwardly.

\medskip

Finally, embedding Chatterjee's method in the GSA framework (called rank-based method in this framework) thereby eliminates the two drawbacks of the classical Pick-Freeze estimation. 
In addition, the rank-based method  allows for the estimation of a large class of GSA indices which include the Sobol indices and the higher order moment indices proposed by Owen \cite{owen2,ODC13,Owen12}. Using a single sample of size $N$, it is now possible to estimate at the same time all the first order Sobol indices , first order Cram\'er-von-Mises indices, and other useful first order sensitivity indices as soon as all inputs are real valued.

\subsection{The universal sensitivity index}\label{ssec:Unvivindex}

Formula \eqref{eq:GIM} can be generalized in the following ways.
\begin{enumerate}
\item The point $a$ in the definition of the test functions is allowed to belong to another measurable space than $ \mathcal{X}^{m}$.
\item The  probability measure $\P^{\otimes m}$  in \eqref{eq:GIM} can be replaced by any ``admissible'' probability measure.
\end{enumerate}
Such generalizations lead to the definition of a universal sensitivity index and its procedures of estimation.

\begin{definition}
Let $a$ belongs to some measurable space $\Omega$ endowed with some probability measure $\Q$. For any $\textbf u\subset \{1,\cdots,p\}$, we define the \emph{universal sensitivity index} with respect to $X_{\textbf u}$ by 
\begin{align}\label{eq:GIMQ}
S_{2,\text{Univ}}^{\textbf u}(T_a,\Q)=\frac{\int_{\Omega}\E\left[\left(\E[T_a(Z)]-\E[T_a(Z)|X_{\textbf u}]\right)^{2}\right]d\Q(a)}{\int_{\Omega} \Var(T_a(Z))d\Q(a)}=\frac{\int_{\Omega}
\Var\left(\E[T_a(Z)|X_{\textbf u}]\right)d\Q(a)}
{\int_{\Omega} \Var(T_a(Z))d\Q(a)}.
\end{align}
\end{definition}

Notice that the index $S_{2,\text{Univ}}^{\textbf u}(T_a,\Q)$ is obtained by the integration over $a$ with respect to $\Q$ of the  Hoeffding decomposition of $T_a(Z)$. Hence,
by construction, this  index lies in $[0,1]$ and shares the same properties as its Sobol counterparts:
\begin{enumerate}
\item the different contributions sum to 1; 
\item it is invariant by translation, by any isometric and by any non-degenerated scaling of $Z$. 
\end{enumerate}

The universality is twofold. First, it allows to consider more general relevant  indices. 
Secondly, this definition encompasses, as particular cases, the classical sensitivity indices. Indeed, 
\begin{itemize}
\item the so-called Sobol index $S^{\textbf{u}}$ with respect to $X_{\textbf{u}}$ is $S_{2,\text{Univ}}^{\textbf{u}}(\operatorname{Id},\P)$;
\item the Cram\'er-von-Mises index $S_{2,\text{CVM}}^{\textbf u}$ with respect to $X_{\textbf{u}}$ is $S_{2,\text{Univ}}^{\textbf{u}}(\ind_{\cdot{} \leqslant a},\P^{\otimes d})$ where $\mathcal{X}=\R^d$ and $\Omega=\mathcal{X}$;
\item the general metric space sensitivity index $S_{2,\text{GMS}}^{\textbf{u}}$ with respect to $X_{\textbf{u}}$ is $S_{2,\text{Univ}}^{\textbf{u}}(T_a,\P^{\otimes m})$ where  $\Omega=\mathcal{X}^{m}$.
\end{itemize}
An example where $\Q$ is different from $\P$ will be considered in Section \ref{sec:GSAWS}.

\paragraph{Estimation}

Here, we assume that  $\Q$ is different from $\P^{\otimes m}$ and we follow the same tracks as for the estimation of $S_{2,\text{GMS}}^{\textbf{u}}$ in Section \ref{ssec:GMSindex}.

\begin{itemize}
\item \underline{First method - Pick-Freeze}. We use the same design of experiment as in the First method of Section \ref{ssec:GMSindex} but instead of considering that the $m$ additional $N$-samples $(W_{l,k})$ for $l\in\{1,\ldots,m\}$ and $k\in\{1,\ldots,N\}$ are drawn with respect to the distribution $\P$ of the output, they are now drawn with respect to the law $\Q$. 
More precisely, one  considers  the following design of experiment consisting in:
\begin{enumerate}
\item a classical Pick-Freeze sample, that is two $N$-samples of $Z$: $(Z_j,Z_j^{\textbf{u}})$, $1\leqslant j\leqslant N$;
\item $m$ $\Q$-distributed $N$-samples $W_{l,k}$, $l\in\{1,\ldots,m\}$ and $k\in\{1,\ldots,N\}$ that are  independent of 
$(Z_j,Z_j^{\textbf{u}})$ for $1\leqslant j\leqslant N$.
\end{enumerate}

The empirical estimator of the numerator of $S_{2,\text{Univ}}^{\textbf{u}}$ is then given by
\begin{align*}
\widehat N_{2,\text{Univ},\text{PF}}^{\textbf{u}}=&\frac{1}{N^m} \sum_{1\leq i_1,\dots,i_{m}\leq N} \biggl[\frac 1N \sum_{j=1}^N  T_{W_{1,i_1},\cdots,W_{m,i_m}}(Z_j)T_{W_{1,i_1},\cdots,W_{m,i_m}}(Z_j^{\textbf u}) \biggr]
\\
&-\frac{1}{N^m} \sum_{1\leq i_1,\dots,i_{m}\leq N} \biggl[\frac{1}{2N} \sum_{j=1}^N  \left(T_{W_{1,i_1},\cdots,W_{m,i_m}}(Z_j)+T_{W_{1,i_1},\cdots,W_{m,i_m}}(Z_j^{\textbf u}) \right)\biggr]^2
\end{align*}
while the one of the denominator is 
\begin{align*}
\widehat D_{2,\text{Univ},\text{PF}}^{\textbf{u}}=&\frac{1}{N^m} \sum_{1\leq i_1,\dots,i_{m}\leq N} \biggl[\frac {1}{2N} \sum_{j=1}^N  \left(T_{W_{1,i_1},\cdots,W_{m,i_m}}(Z_j)^2+T_{W_{1,i_1},\cdots,W_{m,i_m}}(Z_j^{\textbf u})^2\right) \biggr]
\\
&-\frac{1}{N^m} \sum_{1\leq i_1,\dots,i_{m}\leq N} \biggl[\frac{1}{2N} \sum_{j=1}^N  \left(T_{W_{1,i_1},\cdots,W_{m,i_m}}(Z_j)+T_{W_{1,i_1},\cdots,W_{m,i_m}}(Z_j^{\textbf u}) \right)\biggr]^2.
\end{align*}

As previously, it is straightforward  (as soon as the collection $(T_a)_{a\in \mathcal{X}^m}$ forms a Donsker's class of functions) to adapt the proof of Theorem \cite[Theorem 3.8]{GKL18} to prove the asymptotic normality of the estimator.

\item \underline{Second method - U-statistics}. This method is not relevant in this case since $\Q\neq\P^{\otimes d}$.

\item \underline{Third method - Rank-based}. Here, the design of experiment reduces to:
\begin{enumerate}
\item a $N$-sample of $Z$: $Z_j$, $1\leqslant j\leqslant N$;
\item a $N$-sample of $W$ that is $\Q$-distributed: $W_k$, $1\leqslant k\leqslant N$,  independent of $Z_j$, $1\leqslant j\leqslant N$.
\end{enumerate}

The empirical estimator $\widehat S_{2,\text{Univ},\text{Rank}}^{\textbf{u}}$ of $S_{2,\text{Univ}}^{\textbf{u}}$ is then given by the ratio between
\begin{align*}
\widehat N_{2,\text{Univ},\text{Rank}}^{\textbf{u}}=&\frac{1}{N^m} \sum_{1\leq i_1,\dots,i_{m}\leq N} \biggl[\frac 1N \sum_{j=1}^N  T_{W_{i_1},\cdots,W_{i_m}}(Z_j)T_{W_{i_1},\cdots,W_{i_m}}(Z_{N(j)}) \biggr]
\\
&-\frac{1}{N^m} \sum_{1\leq i_1,\dots,i_{m}\leq N} \biggl[\frac{1}{N} \sum_{j=1}^N  T_{W_{i_1},\cdots,W_{i_m}}(Z_j)\biggr]^2
\end{align*}
and  $\widehat D_{2,\text{Univ},\text{Rank}}$ 
\begin{align*}
\frac{1}{N^m} \sum_{1\leq i_1,\dots,i_{m}\leq N} \biggl[\frac {1}{N} \sum_{j=1}^N  T_{W_{i_1},\cdots,W_{i_m}}(Z_j)^2 \biggr]
-\frac{1}{N^m} \sum_{1\leq i_1,\dots,i_{m}\leq N}\biggl[\frac{1}{N} \sum_{j=1}^N  T_{W_{i_1},\cdots,W_{i_m}}(Z_j)\biggr]^2.
\end{align*}
\end{itemize}
We recall that this last method only applies for first order sensitivity indices and real-valued input variables.

\subsection{A sketch of answer to Questions 1 to 3}

In the sequel, we discuss how pertinent choices of the metric, of the class of functions $T_a$ and of the probability measure $\Q$ can provide some partial answers to Questions 1 to 3 raised at the beginning of Section \ref{SIcodes}. 
For instance, in order to answer to Question 1, we can consider that   $\mathcal{X}=\mathcal{M}_q(\R)$ is  the space of probability measures $\mu $ on $\R$ that are $L^q$-functions and we endow this space with the Wasserstein metric $W_q$ (see Section \ref{ssec:wass_metric} for some recalls on Wasserstein metrics). We will propose two  possible approaches  to define interesting sensitivity indices in this framework.
\begin{itemize}
\item In Section \ref{ssec:GSAWS_balls}, we use \eqref{eq:GIM} with $m=2$, $a=(\mu_1,\mu_2)$ and $T_a(Z)=\ind_{\{Z\in B(\mu_1,\mu_2)\}}$ where $B(\mu_1,\mu_2)$ is the open ball: $\{\mu\in\mathcal{M}_q(\R), W_q(\mu,\mu_1)<W_q(\mu_1,\mu_2)\}$.
\item In Section \ref{ssec:GSAWS_frechet}, we use the notion of Fr\'echet means on Wasserstein spaces (see Section \ref{sec:frechet}) and the index defined in \eqref{eq:GIMQ} with  appropriate choices of $a$, $T_a$, and $\Q$. 
\end{itemize}
The case of stochastic computer computer codes raised in Question 2 will be addressed as follows. A computer code (to be defined) valued in $\mathcal M_q(\R)$  will be seen as an ideal case of stochastic computer codes.  
Finally, it will be possible to treat Question 3 using the framework of Question 2.

\section{ Wasserstein spaces and random distributions}\label{sec:wass}

\subsection{Definition}\label{ssec:wass_metric}

For any $q\geqslant 1$, we define the $q$-Wasserstein distance between two probability distributions that are $L^q$-integrable and characterized by their c.d.f.'s $F$ and $G$ on $\mathbb R^d$ by: 
$$ W_q(F,G)=\min_{X\sim F,Y\sim G} \left(\mathbb E[\|X-Y\|^q]^{1/q}\right),$$
where $X\sim F$ and $Y\sim G$ mean that $X$ and $Y$ are random variables with respective c.d.f.'s $F$ and $G$.
We define the Wasserstein space $\mathcal W_q(\R^d)$ as the space of all $L^q$-integrable measures defined on $\R^d$  endowed with the $q$-Wasserstein distance $W_q$. In the sequel, any measure is identified to  its c.d.f. or in some cases to its p.d.f. 
In the unidimensional case ($d=1$), it is a well known fact that $W_q(F,G)$ has an explicitly expression given by
\begin{align}\label{eq:w_1}
W_q(F,G)=\left(\int_0^1|F^-(v)-G^-(v)|^q dv\right)^{1/q}=\mathbb E[|F^-(U)-G^-(U)|^q]^{1/q}.
\end{align} 
Here $F^-$ and $G^-$ are the generalized inverses of the increasing functions $F$ and $G$ and $U$ is a random variable uniformly distributed on $[0,1]$. Of course, $F^-(U)$ and $G^-(U)$ have c.d.f.'s $F$ and $G$.
The representation \eqref{eq:w_1} of the $q$-Wasserstein distance when $d=1$ can be generalized to a wider class of ``contrast functions''. For more details on Wasserstein spaces, one can refer to \cite{ villani2003tot} and \cite{BL18} and the references therein.
\begin{definition}
We call \emph{contrast function} any application $c$  from $\R^{2}$ to $\R$ satisfying the "measure  property" $\cal P$ defined by 
\begin{equation*}
\mathcal{ P}: \forall x\leq x'\ \mathrm{and\ } \forall y\leq y', c(x',y')-c(x',y)-c(x,y')+c(x,y)\leq 0,
\end{equation*}
 meaning that $c$ defines  a negative measure on $\mathbb R^2$.
\end{definition}
For instance, $c(x,y)=-xy$ satisfies $\cal P$.
If $c$ satisfies $\cal P$ then any function of the form $a(x)+b(y)+c(x,y)$ also satisfies $\cal P$.
If $C$ is a convex real function, $c(x,y)=C(x-y)$ satisfies $\cal P$. In particular, $c(x,y)=(x-y)^2=x^2+y^2-2xy$ satisfy $\cal P$ and actually so does $c(x,y)=|x-y|^p$ as soon as $p\geqslant  1$.
\begin{definition}
We define the \emph{Skorohod space} $\mathcal{D}:=\mathcal{D}\left([0,1]\right)$ of all distribution functions as the space of all non-decreasing functions from $\R$ to $[0,1]$ that are c\`ad-l\`ag with limit $0$ (resp. $1$) in $-\infty$ (resp. $+\infty$) equipped with the supremum norm.
\end{definition}
\begin{definition}
For any $F\in\mathcal{D}$, any $G\in\mathcal{D}$,  and any positive contrast  function $c$, we define the \emph{$c$-Wasserstein cost} by 
\[
\displaystyle W_c(F,G)=\min_{X\sim F,Y\sim G}\mathbb E \left[c(X,Y)\right]<+\infty.
\]
\end{definition}
Obviously, $W_q^q=W_c$ with $c(x,y)=|x-y|^q$. 
The following theorem can be found in (\cite{Cambabis76}).
\begin{theorem}[Cambanis, Simon, Stout \cite{Cambabis76}] Let $c$ be a contrast function. Then 
\[
W_c(F,G)=\int_0^1 c(F^{-}(v),G^{-}(v))dv=\E[c(F^{-}(U),G^{-}(U))],
\]
where $U$ is a random variable uniformly distributed on $[0,1]$.
\end{theorem}
%
%
%
%

\subsection{Extension of the Fr\'echet mean to contrast functions}\label{sec:frechet}

\begin{definition}\label{def:frechet_feature}
We call a \emph{loss function} any positive and measurable function $l$. Then, we define a \emph{Fr\'echet feature} ${\cal E}_l[ X]$ of a random variable $X$ taking values in a measurable space ${\cal M}$ as (whenever it exists):
\begin{align}\label{eq:frechet_feature}
{\cal E}_l[ X]\in \argmin_{\theta \in {\cal M} } \mathbb E[l(X,\theta)].
\end{align}
\end{definition}
When $\mathcal M$ is a metric space endowed with a distance $d=l$, the Fr\'echet feature corresponds to the classical Fr\'echet mean (see \cite{fre}). 
In particular, ${\cal E}_d[ X]$ minimizes $\mathbb E[ d^2(X,\theta)]$ which is an extension of the definition of the classical mean in $\R^d$ that minimizes  $\mathbb E[\|X-\theta\|^2]$. 

Now we consider $\mathcal M=\mathcal{D}$ and $l=W_c$. Further, Equation \eqref{eq:frechet_feature} becomes 
\begin{align*}
{\cal E}_{W_c} [ \mathbb F] \in \argmin_{G\in \mathcal{D}}\mathbb E\left[W_c(\mathbb F, G)\right].
\end{align*}
where $\mathbb F$ is a measurable function from a measurable space $\Omega$ to $\mathcal D$.

\begin{theorem} \label{th:frechet_feature} Let $c$ be a positive contrast function. Assume that the application defined  by $(\omega,v)\in\Omega\times(0,1)\mapsto \mathbb F^{-}(\omega,v)\in \R$ is measurable. In addition, assume that ${\cal E}_c [ \mathbb F] $ exists and is unique. Then there exists a unique Fr\'echet mean of $\mathbb E[ c(\mathbb F^{-}(v),s)]$ denoted by 
${\cal E}_c  [\mathbb F^{-}](v)$ and we have
\[
({\cal E}_c  [\mathbb F])^{-}(v) = {\cal E}_c  [\mathbb F^{-}](v) = \argmin_{s\in \mathbb R} \mathbb E[ c(\mathbb F^{-}(v),s)].
\]
\end{theorem}

\begin{proof}[Proof of Theorem \ref{th:frechet_feature}]
Since $c$ satisfies $\cal P$, we have
\[
\mathbb E[W_c(\mathbb F, G)]=\mathbb E\left[\int_0^1 c(\mathbb F^-(v),G^-(v))dv\right]=\int_0^1 \mathbb E[c(\mathbb F^-(v),G^-(v))]dv,
\]
by Fubini's theorem. 
Now, for all $v\in (0,1)$, the quantity  $\mathbb E[c(\mathbb F^-(v),G^-(v))]$ is minimum for $G^-(v)={\cal E}_c [\mathbb F^{-1}](v)$. 
\begin{align*}
\int_0^1 \mathbb E[c(\mathbb F^-(v),{\cal E}_c [\mathbb F^{-1}](v))]dv \leqslant \int_0^1 \mathbb E[c(\mathbb F^-(v),G^-(v))]dv
\end{align*}
and, in particular, for $G^-={\cal E}_c [\mathbb F]^{-1}$, one gets
\begin{align*}
\int_0^1 \mathbb E[c(\mathbb F^-(v),{\cal E}_c [\mathbb F^{-1}](v))]dv \leqslant \int_0^1 \mathbb E[c(\mathbb F^-(v),{\cal E}_c [\mathbb F]^{-1}(v))]dv.
\end{align*}
Conversely, by the definition of ${\cal E}_c [\mathbb F]^{-1}$, we have for all $G$, 
\begin{align*}
\int_0^1 \mathbb E[c(\mathbb F^-(v),{\cal E}_c [\mathbb F]^{-1}(v))]dv &\leqslant \int_0^1 \mathbb E[c(\mathbb F^-(v),G^-(v))]dv
\end{align*}
and, in particular, for $G^-={\cal E}_c [\mathbb F^{-1}]$, one gets
\begin{align*}
\int_0^1 \mathbb E[c(\mathbb F^-(v),{\cal E}_c [\mathbb F]^{-1}(v))]dv &\leqslant \int_0^1 \mathbb E[c(\mathbb F^-(v),{\cal E}_c [\mathbb F^{-1}](v))]dv.
\end{align*}
The theorem then follows by the uniqueness of the minimizer.
\end{proof}

In the previous theorem, we propose a very genereal non  parametric framework for the random element  $\mathbb{F}$ together with some assumptions on existence and uniqueness of the Fr\'echet feature and measurability of the map $(\omega,v)\mapsto \mathbb F^{-}(\omega,v)$. Nevertheless, it is possible to  construct explicit parametric models for $\mathbb{F}$ for whom theses assumptions are satisfied. For instance, the authors of \cite{BK2015} ensures measurability for some parametric models on $\mathbb{F}$ using results of \cite{MR2643592}.
Notice that in \cite{FKR2013} a new sensitivity indice is defined for each feature associated to a contrast function. In section \ref{ssec:GSAWS_frechet} we will restrict our analysis to  Fr\'echet means, hence to Sobol indices.

\subsection{Examples}

The Fr\'echet mean in the $\mathcal W_2(\R)$ space is the inverse of the function $v\mapsto \mathbb E\ \left[ \mathbb F^{-}(v)\right]$.
Another example is the Fr\'echet median. Since the median in $\R$ is related to the $L^1$ cost, the Fr\'echet $\mathcal W_1(\R)$ median of a random c.d.f.\ is
\[
(\mbox{Med} (\mathbb F)^{-}(v)\in \mbox{Med}(  \mathbb F^{-}(v)).
\]

More generally, we recall that, for $\alpha\in (0,1)$, the $\alpha$-quantile in $\R$ 
is the Fr\'echet mean associated to the contrast function
$c_{\alpha}(x,y)=(1-\alpha)(y-x) \ind_{x-y<0}+\alpha (x-y) \ind_{x-y\geqslant  0}$, also called the \emph{pinball function}. Then we
can define an $\alpha$-quantile $q_\alpha(\mathbb F)$ of a random c.d.f.\ as:
\[
(q_\alpha (\mathbb F))^{-}(v)\in q_\alpha(  \mathbb F^{-}(v)),
\]
where $q_\alpha(X)$ is the set of the $\alpha$-quantiles of a random variable $X$ taking values in $\mathbb R$. Naturally, taking $\alpha=1/2$ leads to the median.

Let us illustrate the previous definitions on an example.
Let $X$ be a random variable with c.d.f. $F_0$  which is assumed to be  increasing and continuous.
Let also $m$ and  $\sigma$ two real random variables such that $\sigma$>0. Then we consider the random c.d.f.\ $\mathbb F$ of $\sigma X+m$:  
\[
\mathbb F(x)= F_0\left(\frac{x- m}{\sigma}\right) \quad \text{and} \quad \mathbb F^{-1}(v)=\sigma F_0^{-1}(v)+m.
\]
Naturally, the Fr\'echet mean of $\mathbb F$ is ${\cal E} [\mathbb F](x)=F_0\left({x-\E[ m]}/{\E[\sigma]}\right)$ and its $\alpha$-quantile  is given by 
\[
(q_\alpha (\mathbb F))^{-1}(v)=q_\alpha(  \sigma F_0^{-1}(v)+m).\]

\section{Sensitivity analysis in general Wasserstein spaces}\label{sec:GSAWS}

In this section, we consider that our computer code is $\mathcal W_q(\R)$-valued; namely, the output of an experiment is the c.d.f.\ or the p.d.f.\ of a measure $\mu\in \mathcal W_q(\R)$.  
For instance, in \cite{BILGLR16}, \cite{LG13} and \cite{MNP15}, the authors deal with p.d.f.-valued computer codes (and stochastic computer codes). In other words, they define the following application:
\begin{eqnarray}\label{def:code_density}
f  : & E  & \to  \mathcal F\\
& x & \mapsto  f(x)\nonumber
\end{eqnarray} 
where $\mathcal F$ is the set of p.d.f.'s:
\[
\mathcal F= \left\{g\in L^1(\R);\ g\geqslant 0,\ \int_{\R} g(x)dx=1\right\}.
\]

Here, we choose to identify any element of $\mathcal W_q(\R)$ with its c.d.f.  In this framework, the output of the cmoputer code is then a c.d.f.\ denoted by
\begin{equation}\label{eq:modelW}
\mathbb{F}=f(X_1,\ldots,X_p).
\end{equation}
Here, $\P$ denotes the law of the c.d.f.\ $\F$. In addition, we set $q=2$. 
The case of a general $q$ can be handled analogously.

\subsection{Sensitivity anlaysis using Equation \texorpdfstring{\eqref{eq:GIM}}{f} and Wasserstein balls}\label{ssec:GSAWS_balls}

Consider $F$, $F_1$, and $F_2$ three elements of  $\mathcal W_2 (\R)$ and, for $a=(F_1,F_2)$, the family of test functions
\begin{equation}\label{def:Ta_W}
T_a(F)=T_{(F_1,F_2)}(F)=\1_{W_2 (F_1,F)\leqslant W_2 (F_1,F_2)}.
\end{equation}
Then, for all $\textbf u\subset \{1,\cdots,p\}$,  \eqref{eq:GIM} becomes 
\begin{align}\label{eq:SobolW}
S_{2,W_2 }^{\textbf u}&= S_{2,\text{Univ}}^{\textbf u}((F_1,F_2,F)\mapsto T_{F_1,F_2}(F),\P^{\otimes 2})\nonumber\\
&=\frac{\int_{\mathcal W_2 (\R)\times \mathcal W_2 (\R)}\E\left[\left(\E[\1_{W_2 (F_1,\mathbb{F})\leqslant W_2 (F_1,F_2)}]-\E[\1_{W_2 (F_1,\F)\leqslant W_2 (F_1,F_2)}|X^{\textbf u}]\right)^{2}\right]d\P^{\otimes 2}(F_1,F_2)}{\int_{\mathcal W_2 (\R)\times \mathcal W_2 (\R)} \Var(\1_{W_2 (F_1,\mathbb{F})\leqslant W_2 (F_1,F_2)})d\P^{\otimes 2}(F_1,F_2)}\nonumber\\
&=\frac{\int_{\mathcal W_2 (\R)\times \mathcal W_2 (\R)}\Var\left(\E[\1_{W_2 (F_1,\mathbb{F})\leqslant W_2 (F_1,F_2)}\vert X^{\bf u}]\right) d\P^{\otimes 2}(F_1,F_2)}{\int_{\mathcal W_2 (\R)\times \mathcal W_2 (\R)} \Var(\1_{W_2 (F_1,\mathbb{F})\leqslant W_2 (F_1,F_2)})d\P^{\otimes 2}(F_1,F_2)}.
\end{align}
As explained in Section \ref{ssec:GMSindex}, $S_{2,W_2 }^{\textbf u}$ is obtained by integration over $a$ of the  Hoeffding decomposition of $T_a(\F)$. Hence,
by construction, this  index lies in $[0,1]$ and shares the same properties as its Sobol counterparts:
\begin{enumerate}
\item the different contributions sum to 1; 
\item it is invariant by translation, by any isometric and by any non-degenerated scaling of $\F$. 
\end{enumerate}

\subsection{Sensitivity analysis using Equation \texorpdfstring{\eqref{eq:GIMQ}}{f} and Fr\'echet means}\label{ssec:GSAWS_frechet}

In the classical framework where the output $Z$ is real, we recall that the Sobol index with respect to $X_{\textbf u}$ is defined by 
\begin{equation}\label{sobol1}
S^{\textbf u}=\frac{\text{Var}(\mathbb E[Z|X_{\textbf u}])}{\text{Var}(Z)}=\frac{\Var(Z)-\mathbb E[\Var(Z|X_{\textbf u})]}{\Var(Z)},
\end{equation}
by the property of the conditional expectation. 
In one hand, one may extend this formula to the framework of this section where the output of interest is the c.d.f. $\mathbb F$:
\[
S^{\textbf u}(\mathbb F)=\frac{\Var(\mathbb F)-\mathbb E[\Var(\mathbb F|X_{\textbf u}))]}{\Var (\mathbb F)},
\]
where $\Var(\mathbb F)=\E[ W_2^2(\mathbb F,{\cal E}_{W_2 }(\mathbb F))]$ with ${\cal E}_{W_2 }(\mathbb F)$ the Fr\'echet mean of $\F$.
From Theorem \ref{th:frechet_feature}, we get
\[
\Var(\F)=\E\left[\int_0^1 |\F^-(v)-{\cal E}(\F)^-(v)|^2dv\right]=\E\left[\int_0^1 |\F^-(v)-\E[\F^-(v)]|^2dv\right]=\int_0^1\Var(\F^-(v))dv
\]
 leading to 
 \begin{align}
 S^{\textbf u}(\F)&=\frac{\int_0^1\Var(\F^-(v))dv-\int_0^1\E [\Var(\F^-(v)|X_{\textbf u})]dv}{\int_0^1\Var(\mathbb F^-(v))dv}=\frac{\int_0^1\Var(\E[\F^-(v)|X_{\textbf u}])dv}{{\int_0^1\Var(\F^-(v))dv}}.\label{eq:SobolW2} 
 \end{align}
 In the other hand, one can consider  \eqref{eq:GIMQ}, with 
  $m=1$, 
\begin{align}\label{def:Ta_F}
T_v(\mathbb F)=\mathbb F^-(v)
\end{align}
and $\Q$  the uniform probability measure on $[0,1]$. In that case,
\begin{equation*}
\Var(\F)=\E\left[\int_0^1 |\F^-(v)-{\cal E}_{W_2 }(\F)^-(v)|^2dv\right]=\int_0^1\Var(\F^-(v))dv=\E [W_2 ^2(\F,{\cal E}_{W_2 }(\F))]. 
\end{equation*}
  Then 
  \begin{equation*}
S_{2,\text{Univ}}^{\textbf u}(T_v,\mathcal U([0,1]))=\frac{\int_0^1\E\left[\left({\cal E}_{W_2 }(\F)^-(v)-{\cal E}_{W_2 }(\F|X_{\textbf u})^-(v)\right)^{2}\right]dv}{\int_0^1\Var(\F^-(v))dv}=\frac{\E\left[ W_2 ^2({\cal E}_{W_2 }(\F|X_{\textbf u}),{\cal E}_{W_2 }(\F))\right]}{\E \left[W_2 ^2(\F,{\cal E}_{W_2 }(\F))\right]}. 
 \end{equation*}
 is exactly the same as $S^{\textbf u}(\mathbb F)$ in \eqref{eq:SobolW2}.
Thus, as explained in Section \ref{ssec:Unvivindex}, $S^{\textbf u}(\mathbb F)$ lies in $[0,1]$ and:
\begin{enumerate}
\item the different contributions sum to 1;
\item it is invariant by translation, by any isometric and by any non-degenerated scaling of $\F$.
\end{enumerate}

\subsection{Estimation procedure}\label{ssec:GSAWS_est}

As noticed in the previous section, both 
\[
S_{2,W_2 }^{\textbf u}= S_{2,Univ}^{\textbf u}(T_a,\P^{\otimes 2})
\]
with $T_a$ defined in \eqref{def:Ta_W} and 
\[
S^{\textbf u}(\mathbb F)=S_{2,\text{Univ}}^{\textbf u}(T_v,\mathcal U([0,1]))
\] 
with $T_v$ defined in \eqref{def:Ta_F}, are particular cases of indices of the form \eqref{eq:GIMQ}. 

When $a$ belongs to the same space as the output and when $\Q$ is equal to $\P^{\otimes m}$, one may first use the Pick-Freeze estimations of the indices given in \eqref{eq:SobolW}
and \eqref{eq:SobolW2}. To do so, it is convenient once again to use \eqref{eq:trick_PF} leading to
\begin{align}\label{eq:SobolW_PF}
S_{2,W_2 }^{\textbf u}
&=\frac{\int_{\mathcal W_2 (\R)\times \mathcal W_2 (\R)}\Cov\left(
\1_{W_2 (F_1,\mathbb{F})\leqslant W_2 (F_1,F_2)},
\1_{W_2 (F_1,\mathbb{F}^{\textbf u})\leqslant W_2 (F_1,F_2)}
\right) d\P^{\otimes 2}(F_1,F_2)}{\int_{\mathcal W_2 (\R)\times \mathcal W_2 (\R)} \Var(\1_{W_2 (F_1,\mathbb{F})\leqslant W_2 (F_1,F_2)})d\P^{\otimes 2}(F_1,F_2)}
\end{align}
and
\begin{align}
 S^{\textbf u}(\F)&=\frac{\int_0^1\Cov\left(\F^-(v),\F^{-,\textbf u}(v)\right)dv}{{\int_0^1\Var(\F^-(v))dv}}\label{eq:SobolW2_PF} 
\end{align}
where $\mathbb{F}^{\textbf u}$ and $\F^{-,\textbf u}$ are respectively the Pick-Freeze versions of $\mathbb{F}$ and $\F^{-}$.
Secondly, one may resort to the estimations  based on U-statistics together on the Pick-Freeze design of experiment.
Thirdly, it is also possible and easy to obtain rank-based estimations in the vein of  \eqref{eq:coeffChatterjee}.

\subsection{Numerical comparison of both indices}

\begin{ex}[Toy model]\label{ex:toy_model}
Let $X_1,X_2,X_3$ be $3$ independent and positive random variables. 
We consider the c.d.f.-valued code $f$, the output of which is given by 
\begin{equation}\label{def:toy_model_1_interaction}
\mathbb{F}(t)=\frac{t}{1+X_1+X_2+X_1X_3}\1_{0\leqslant t\leqslant 1+X_1+X_2+X_1X_3}+ \1_{1+X_1+X_2+X_1X_3<t},
\end{equation} 
so that
\begin{equation}
\mathbb{F}^{-1}(v)=v\Bigl(1+X_1+X_2+X_1X_3\Bigr).
\end{equation} 
In addition, one gets 
\begin{align*}
\Var\left(\mathbb{F}^{-1}(v)\right)&=v^2\left(\Var(X_1(1+X_3))+\Var(X_2)\right)\\
&=v^2\left(\Var(X_1)\Var(X_3)+\Var(X_1)(1+\E[X_3])^2+\Var(X_3)\E[X_1]^2+\Var(X_2)\right)\\
\end{align*}
and 
\begin{align*}
&\E\left[\mathbb{F}^{-1}(v)|X_1\right]=v\Bigl(1+X_1(1+\E[X_3])+\E[X_2]\Bigr),\\
&\E\left[\mathbb{F}^{-1}(v)|X_2\right]=v\Bigl(1+\E[X_1](1+\E[X_3])+X_2\Bigr),\\
&\E\left[\mathbb{F}^{-1}(v)|X_3\right]=v\Bigl(1+\E[X_1](1+X_3)+\E[X_2]\Bigr),\\
&\E\left[\mathbb{F}^{-1}(v)|X_1X_3\right]=v\Bigl(1+X_1(1+X_3)+\E[X_2]\Bigr),
\end{align*}
and finally
\begin{align*}
&\Var\left(\E\left[\mathbb{F}^{-1}(v)|X_1\right]\right)=v^2(1+\E[X_3])^2\Var(X_1),\\
&\Var\left(\E\left[\mathbb{F}^{-1}(v)|X_2\right]\right)=v^2\Var(X_2),\\
&\Var\left(\E\left[\mathbb{F}^{-1}(v)|X_3\right]\right)=v^2\E[X_1]^2\Var(X_3),\\
&\Var\left(\E\left[\mathbb{F}^{-1}(v)|X_1,X_3\right]\right)
=v^2\left(\Var(X_1)\Var(X_3)+\Var(X_1)(1+\E[X_3])^2+\Var(X_3)\E[X_1]^2\right).
\end{align*}
For ${\bf u}  =i\in{1,2,3}$ or ${\bf u}=\{1,3\}$, it remains to plug the previous formulas in \eqref{eq:SobolW2} to get the explicit expressions of the indices $S^{\bf u}(\mathbb{F})$.

Now, in order to get a closed formula for the indices defined in \eqref{eq:SobolW},  we assume $X_i$ is Bernoulli distributed with parameter $0<p_i<1$.
In \eqref{eq:SobolW}, the distributions $F_1$ and $F_2$   can be either $\mathcal{U}([0,1])$, $\mathcal{U}([0,2])$,  $\mathcal{U}([0,3])$, or $\mathcal{U}([0,4])$ with respective probabilities $q_1=(1-p_1)(1-p_2)$, $q_2=(1-p_1)p_2+p_1(1-p_2)(1-p_3)$, $q_3=p_1((1-p_2)p_3+p_2(1-p_3))$, and $q_4=p_1p_2p_3$. In the sequel, we give, for all sixteen possibilities for the distribution of $(F_1,F_2)$, the corresponding contributions for the numerator and for the denominator of \eqref{eq:SobolW}.

With probability $p_{1,1}=(1-p_1)^2(1-p_2)^2$,  $F_1$ and $F_2\sim \mathcal{U}([0,1])$. Then $W_2^2(F_1,F_2)=0$, 
$W_2^2(F_1,\mathbb{F})=\frac13(X_1+X_2+X_1X_3)^2$, and 
$W_2^2(F_1,\mathbb{F})\leqslant W_2^2(F_1,F_2)$ if and only if $X_1+X_2+X_1X_3=0$. 
Since $\P\left(X_1+X_2+X_1X_3=0\right)=(1-p_1)(1-p_2)$, the contribution $d_{1,1}$ of this case to the denominator is thus
\begin{equation*}
d_{1,1}=q_{1,1}(1-q_{1,1}) \quad \text{with $q_{1,1}=(1-p_1)(1-p_2)$}.
\end{equation*}
Moreover,
 \begin{align*}
 &\E[\ind_{X_1+X_2+X_1X_3= 0}|X_1]=\P\Bigl(X_1+X_2+X_1X_3= 0|X_1\Bigr)=\ind_{X_1=0}\P(X_2=0)=(1-p_2)\ind_{X_1=0}.
  \end{align*}
so that, the contribution to the numerator is here given by
\begin{equation*}
 n_{1,1}^1= \Var(\E[\ind_{X_1+X_2+X_1X_3= 0}|X_1])=p_1(1-p_1)(1-p_2)^2.
\end{equation*}
Similarly,  one gets 
\[
n_{1,1}^2= \Var(\E[\ind_{X_1+X_2+X_1X_3= 0}|X_2])=p_2(1-p_2)(1-p_1)^2 \quad \text{and} \quad n_{1,1}^3= 0.
\]
Moreover, regarding the indices with respect to $X_1$ and $X_3$, 
\[
\E[\ind_{X_1+X_2+X_1X_3= 0}|X_1,X_3]=\P\Bigl(X_1+X_2+X_1X_3= 0|X_1,X_3\Bigr)=\ind_{X_1=0}\P(X_2=0)=(1-p_2)\ind_{X_1=0}
\]
and the contribution to the numerator is given by
\[
n_{1,1}^{1,3}=\Var(\E[\ind_{X_1+X_2+X_1X_3= 0}|X_1,X_3])=p_1(1-p_1)(1-p_2)^2.
\]
The remaining fifteen cases can be treated similarly and are gathered (with the first case developed above) in the following table. Finally, for $k=1,\ldots,3$, one may compute the explicit expression of $S_{2,W_2}^k$:
\begin{align*}
S_{2,W_2}^k&=\frac{\int_{\mathcal W_2 (\R)\times \mathcal W_2 (\R)}\Cov\left(
\1_{W_2 (F_1,\mathbb{F})\leqslant W_2 (F_1,F_2)},
\1_{W_2 (F_1,\mathbb{F}^{\textbf u})\leqslant W_2 (F_1,F_2)}
\right) d\P^{\otimes 2}(F_1,F_2)}{\int_{\mathcal W_2 (\R)\times \mathcal W_2 (\R)} \Var(\1_{W_2 (F_1,\mathbb{F})\leqslant W_2 (F_1,F_2)})d\P^{\otimes 2}(F_1,F_2)}=\frac{\sum_{i,j} p_{i,j} n_{i,j}^k}{\sum_{i,j} p_{i,j} d_{i,j}}.
\end{align*}
Some numerical values have not been explicited in the table but given below:
\begin{small}
\begin{eqnarray*}
& \text{Case 2} 
& \Var(\ind_{X_1=1}(1-(1-p_2)\ind_{X_3=0}))
=p_1(1-p_1)(1-(1-p_2)(1-p_3))^2+p_1(1-p_2)^2p_3(1-p_3),\\
&\\
& \text{Case 6} 
& \Var(\ind_{X_1=1}(p_2-(1-p_2)\ind_{X_3=0}))
=p_1(1-p_1)(p_2-(1-p_2)(1-p_3))^2+p_1(1-p_2)^2p_3(1-p_3),\\
&\\
& \text{Case 11}
& \Var(\ind_{X_1=1}(p_2+(1-2p_2)\ind_{X_3=1}))
=p_1(1-p_1)(p_2+(1-2p_2)p_3)^2+p_1(1-2p_2)^2p_3(1-p_3),\\
&\\
& \text{Case 15}
 & \Var(\ind_{X_1=1}(p_2+(1-p_2)\ind_{X_3=1}))
 =p_1(1-p_1)(p_2+(1-p_2)p_3)^2+p_1(1-p_2)^2p_3(1-p_3).
\end{eqnarray*}
\end{small}
\begin{table}[htp!]
\begin{center}
\begin{footnotesize}
\begin{tabular}{ll|ll}
\hline
Case 1 & $F_1\sim \mathcal{U}([0,1])$, $F_2\sim\mathcal{U}([0,1])$ & Case 2 & $F_1\sim \mathcal{U}([0,1])$, $F_2\sim\mathcal{U}([0,2])$\\
\hline
\multicolumn{2}{c|}{} & \multicolumn{2}{c}{}\\
Prob.  & $q_1^2$ & Prob. & $q_1q_2$\\
Num. 1  & $p_1(1-p_1)(1-p_2)^2$ & Num. 1 & $p_1(1-p_1)(p_2+p_3-p_2p_3)^2$\\
Num. 2 & $(1-p_1)^2p_2(1-p_2)$ & Num. 2 & $p_1^2p_2(1-p_2)(1-p_3)^2$ \\
Num. 3 & $0$ & Num. 3 & $p_1^2(1-p_2)^2p_3(1-p_3)$ \\
Num. 1,3 & $p_1(1-p_1)(1-p_2)^2$ & Num. 1,3 & $\Var(\ind_{X_1=1}(1-(1-p_2)\ind_{X_3=0})$ \\
$q$ Den. & $(1-p_1)(1-p_2)$ & $q$ Den. & $(1-p_1)+p_1(1-p_2)(1-p_3)$\\
\multicolumn{2}{c|}{} & \multicolumn{2}{c}{}\\
\hline
Case 3 & $F_1\sim \mathcal{U}([0,1])$, $F_2\sim\mathcal{U}([0,3])$ & Case 4 & $F_1\sim \mathcal{U}([0,1])$, $F_2\sim\mathcal{U}([0,4])$\\
\hline
\multicolumn{2}{c|}{} & \multicolumn{2}{c}{}\\
Prob.  & $q_1q_3$ & Prob. & $q_1q_4$\\
Num. 1  & $p_1(1-p_1)p_2^2p_3^2$ & Num. 1 & $0$\\
Num. 2 & $p_1^2p_2(1-p_2)p_3^2$ & Num. 2 & $0$ \\
Num. 3 & $p_1^2p_2^2p_3(1-p_3)$ & Num. 3 & $0$ \\
Num. 1,3 & $p_1p_2^2p_3(1-p_1p_3)$ & Num. 1,3 & $0$ \\
$q$ Den. & $1-p_1p_2p_3$ & $q$ Den. & $0$\\
\multicolumn{2}{c|}{} & \multicolumn{2}{c}{}\\
\hline
Case 5 & $F_1\sim \mathcal{U}([0,2])$, $F_2\sim\mathcal{U}([0,1])$ & Case 6 & $F_1\sim \mathcal{U}([0,2])$, $F_2\sim\mathcal{U}([0,2])$\\
\hline
\multicolumn{2}{c|}{} & \multicolumn{2}{c}{}\\
Prob.  & $q_1q_2$ & Prob. & $q_2^2$\\
Num. 1  & $p_1(1-p_1)p_2^2p_3^2$ & Num. 1 & $p_1(1-p_1)(p_2-(1-p_2)(1-p_3))^2$\\
Num. 2 & $p_1^2p_2(1-p_2)p_3^2$ & Num. 2 & $p_2(1-p_2)(p_1(1-p_3)-(1-p_1))^2$ \\
Num. 3 & $p_1^2p_2^2p_3(1-p_3)$ & Num. 3 & $p_1^2(1-p_2)^2p_3(1-p_3)$ \\
Num. 1,3 & $p_1p_2^2p_3(1-p_1p_3)$ & Num. 1,3 & $\Var(\ind_{X_1=1}(p_2-(1-p_2)\ind_{X_3=0}))$\\
$q$ Den. & $1-p_1p_2p_3$ & $q$ Den. & $(1-p_1)p_2+p_1(1-p_2)(1-p_3)$\\
\multicolumn{2}{c|}{} & \multicolumn{2}{c}{}\\
\hline
Case 7 & $F_1\sim \mathcal{U}([0,2])$, $F_2\sim\mathcal{U}([0,3])$ & Case 8 & $F_1\sim \mathcal{U}([0,2])$, $F_2\sim\mathcal{U}([0,4])$\\
\hline
\multicolumn{2}{c|}{} & \multicolumn{2}{c}{}\\
Prob.  & $q_2q_3$ & Prob. & $q_2q_4$\\
Num. 1  & $p_1(1-p_1)p_2^2p_3^2$ & Num. 1 & $0$\\
Num. 2 & $p_1^2p_2(1-p_2)p_3^2$ & Num. 2 & $0$ \\
Num. 3 & $p_1^2p_2^2p_3(1-p_3)$ & Num. 3 & $0$ \\
Num. 1,3 & $p_1p_2^2p_3(1-p_1p_3)$ & Num. 1,3 & $0$ \\
$q$ Den. & $1-p_1p_2p_3$ & $q$ Den. & $0$\\
\multicolumn{2}{c|}{} & \multicolumn{2}{c}{}\\
\hline
Case 9 & $F_1\sim \mathcal{U}([0,3])$, $F_2\sim\mathcal{U}([0,1])$ & Case 10 & $F_1\sim \mathcal{U}([0,3])$, $F_2\sim\mathcal{U}([0,2])$\\
\hline
\multicolumn{2}{c|}{} & \multicolumn{2}{c}{}\\
Prob.  & $q_1q_3$ & Prob. & $q_2q_3$\\
Num. 1  & $0$ & Num. 1 & $p_1(1-p_1)(1-p_2)^2$\\
Num. 2 & $0$ & Num. 2 & $(1-p_1)^2p_2(1-p_2)$ \\
Num. 3 & $0$ & Num. 3 & $0$ \\
Num. 1,3 & $0$ & Num. 1,3 & $p_1(1-p_1)(1-p_2)^2$ \\
$q$ Den. & $0$ & $q$ Den. & $(1-p_1)p_2+p_1$\\
\multicolumn{2}{c|}{} & \multicolumn{2}{c}{}\\
\hline
Case 11 & $F_1\sim \mathcal{U}([0,3])$, $F_2\sim\mathcal{U}([0,3])$ & Case 12 & $F_1\sim \mathcal{U}([0,3])$, $F_2\sim\mathcal{U}([0,4])$\\
\hline
\multicolumn{2}{c|}{} & \multicolumn{2}{c}{}\\
Prob.  & $q_3^2$ & Prob. & $q_3q_4$\\
Num. 1  & $p_1(1-p_1)(p_2(1-p_3)+(1-p_2)p_3)^2$ & Num. 1 & $p_1(1-p_1)(1-p_2)^2$\\
Num. 2 & $p_1^2p_2(1-p_2)(2p_3-1)^2$ & Num. 2 & $(1-p_1)^2p_2(1-p_2)$ \\
Num. 3 & $p_1^2(2p_2-1)^2p_3(1-p_3)$ & Num. 3 & $0$ \\
Num. 1,3 & $\Var(\ind_{X_1=1}(p_2+(1-2p_2)\ind_{X_3=1})$ & Num. 1,3 & $p_1(1-p_1)(1-p_2)^2$ \\
$q$ Den. & $p_1(p_2(1-p_3)+(1-p_2)p_3)$ & $q$ Den. & $(1-p_1)p_2+p_1$\\
\multicolumn{2}{c|}{} & \multicolumn{2}{c}{}\\
\hline
Case 13 & $F_1\sim \mathcal{U}([0,4])$, $F_2\sim\mathcal{U}([0,1])$ & Case 14 & $F_1\sim \mathcal{U}([0,4])$, $F_2\sim\mathcal{U}([0,2])$\\
\hline
\multicolumn{2}{c|}{} & \multicolumn{2}{c}{}\\
Prob.  & $q_1q_4$ & Prob. & $q_2q_4$\\
Num. 1  & $0$ & Num. 1 & $p_1(1-p_1)(1-p_2)^2$\\
Num. 2 & $0$ & Num. 2 & $(1-p_1)^2p_2(1-p_2)$ \\
Num. 3 & $0$ & Num. 3 & $0$ \\
Num. 1,3 & $0$ & Num. 1,3 & $p_1(1-p_1)(1-p_2)^2$ \\
$q$ Den. & $0$ & $q$ Den. & $(1-p_1)p_2+p_1$\\
\multicolumn{2}{c|}{} & \multicolumn{2}{c}{}\\
\hline
Case 15 & $F_1\sim \mathcal{U}([0,4])$, $F_2\sim\mathcal{U}([0,3])$ & Case 16 & $F_1\sim \mathcal{U}([0,4])$, $F_2\sim\mathcal{U}([0,4])$\\
\hline
\multicolumn{2}{c|}{} & \multicolumn{2}{c}{}\\
Prob.  & $q_3q_4$ & Prob. & $q_4^2$\\
Num. 1  & $p_1(1-p_1)(p_2+(1-p_2)p_3)^2$ & Num. 1 & $p_1(1-p_1)p_2^2p_3^2$\\
Num. 2 & $p_1^2p_2(1-p_2)(1-p_3)^2$ & Num. 2 & $p_1^2p_2(1-p_2)p_3^2$ \\
Num. 3 & $p_1^2(1-p_2)^2p_3(1-p_3)$ & Num. 3 & $p_1^2p_2^2p_3(1-p_3)$ \\
Num. 1,3 & $\Var(\ind_{X_1=1}(p_2+(1-p_2)\ind_{X_3=1})$ & Num. 1,3 & $p_1p_2^2p_3(1-p_1p_3)$ \\
$q$ Den. & $p_1(p_2+(1-p_2)p_3)$ & $q$ Den. & $p_1p_2p_3$\\
\multicolumn{2}{c|}{} & \multicolumn{2}{c}{}\\
\hline
\end{tabular}
\end{footnotesize}
\end{center}
\end{table}

In Figure \ref{fig:toy1_interaction_F}, we have represented the indices 
$S^1(\mathbb{F})$, $S^2(\mathbb{F})$,  $S^3(\mathbb{F})$, and  $S^{13}(\mathbb{F})$ given by \eqref{eq:SobolW2} with respect to the values of $p_1$ and $p_2$ varying from 0 to 1 for a fixed value of $p_3$. We have considered three different values of $p_3$: $p_3=0.01$ (first row), $0.5$, (second row) and $0.99$ (third row). Analogously, the same kind of illustration for the indices $S^1_{2,W_2}$, $S^2_{2,W_2}$,  $S^3_{2,W_2}$, and  $S^{13}_{2,W_2}$ given by \eqref{eq:SobolW} is provided in Figure \ref{fig:toy1_interaction_W} . In addition, the regions of predominance of each index $S^{\textbf u} (\mathbb{F})$ are plotted in Figure \ref{fig:toy1_zone_interaction_F}. The values of $p_1$ and $p_2$ still vary from 0 to 1 and the fixed values of $p_3$ considered are: $p_3=0.01$ (first row), $0.5$, (second row) and $0.99$ (third row). Finally, the same kind of illustration for the indices $S^{\textbf u}_{2,W_2}$ is given in Figure \ref{fig:toy1_zone_interaction_W}.
\begin{figure}[h!]
\centering 
\includegraphics[width=15.8cm]{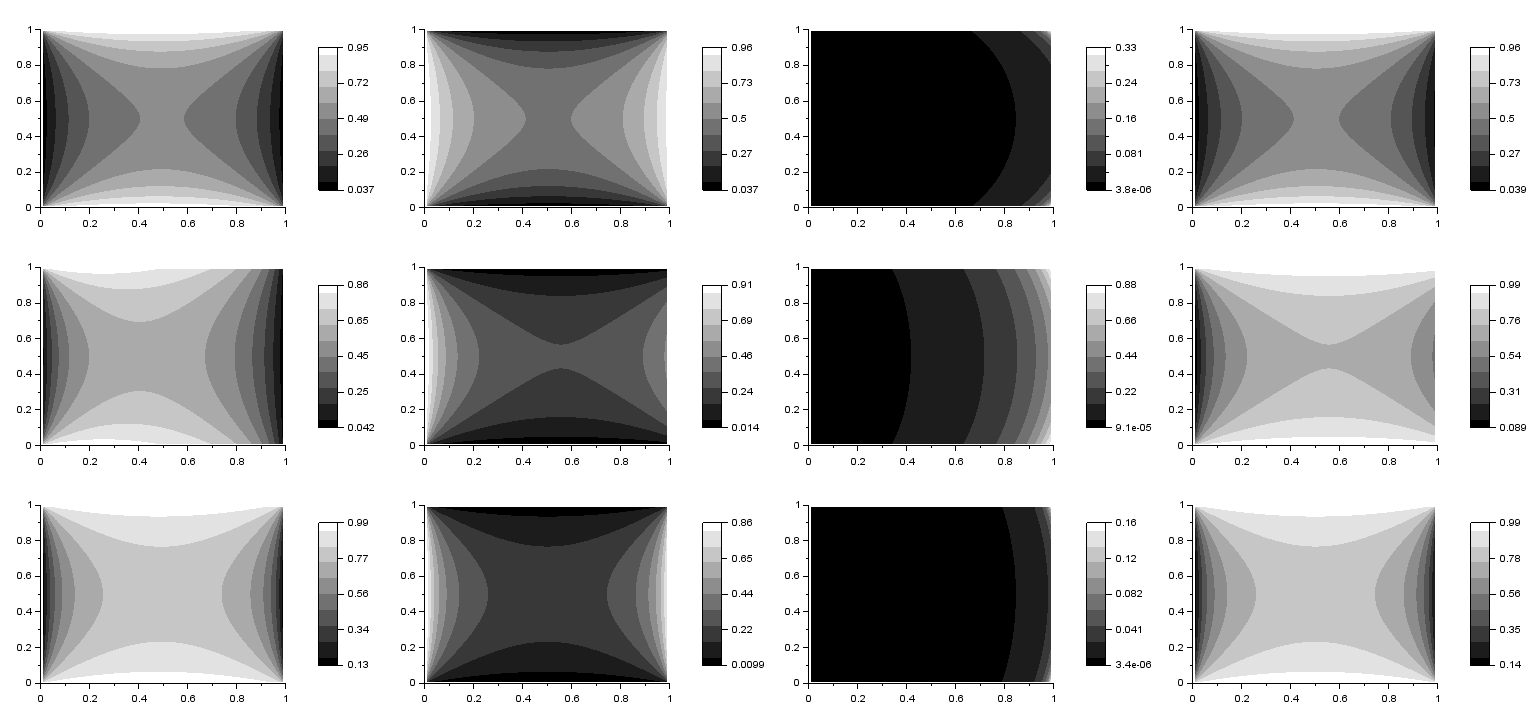}
\caption{Model \eqref{def:toy_model_1_interaction}. Values of the indices $S^1(\mathbb{F})$, $S^2(\mathbb{F})$,  $S^3(\mathbb{F})$, and  $S^{13}(\mathbb{F})$ given by \eqref{eq:SobolW2} (from left to right) with respect to the values of $p_1$ and $p_2$ (varying from 0 to 1). In the first row (resp. second and third), $p_3$ is fixed to $p_3=0.01$ (resp. $0.5$ and $0.99$).}
\label{fig:toy1_interaction_F}
\end{figure}
\begin{figure}[h!]
\centering 
\includegraphics[width=16cm]{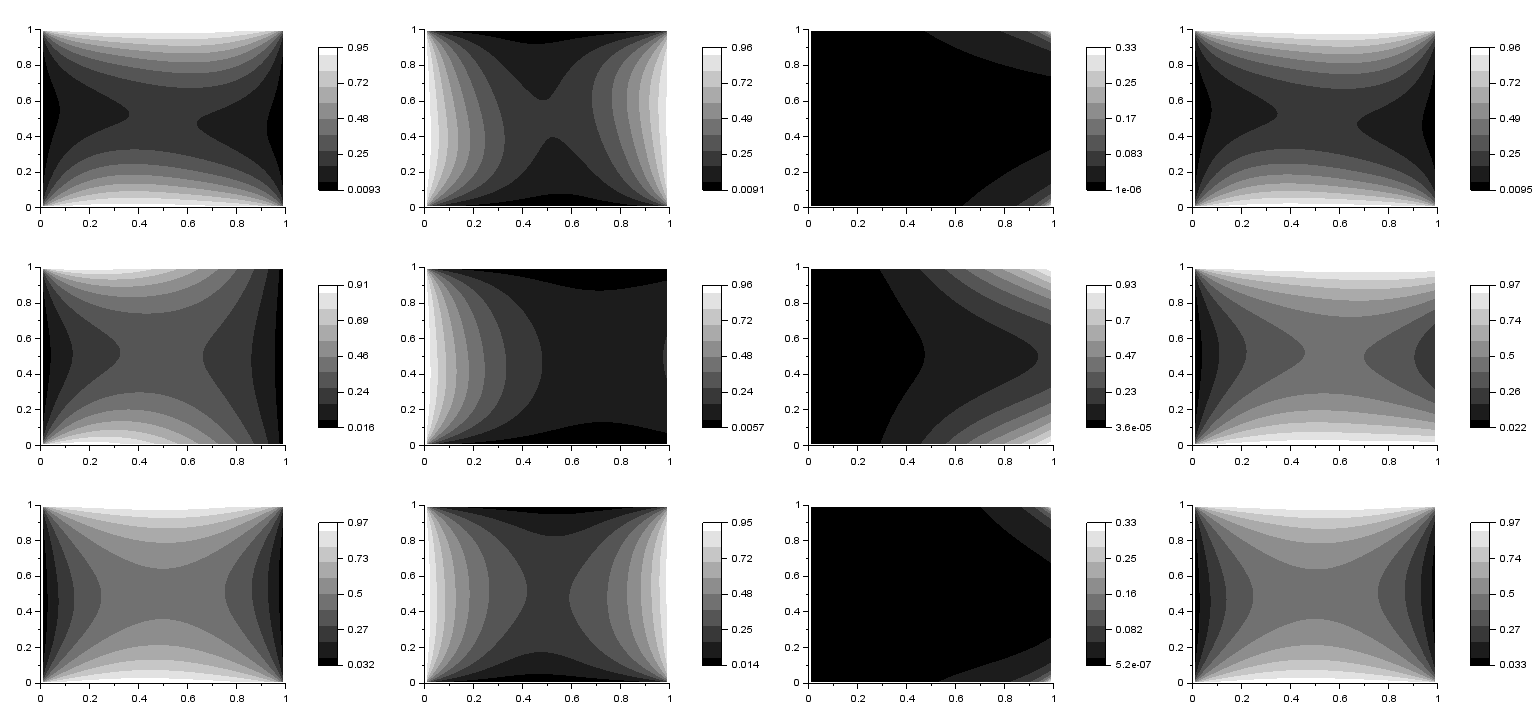}
\caption{Model \eqref{def:toy_model_1_interaction}. Values of the indices $S^1_{2,W_2}$, $S^2_{2,W_2}$,  $S^3_{2,W_2}$, and  $S^{13}_{2,W_2}$ given by \eqref{eq:SobolW} (from left to right) with respect to the values of $p_1$ and $p_2$ (varying from 0 to 1). In the first row (resp. second and third), $p_3$ is fixed to $p_3=0.01$ (resp. $0.5$ and $0.99$).}
\label{fig:toy1_interaction_W}
\end{figure}
\begin{figure}[h!]
\centering 
\includegraphics[width=16cm]{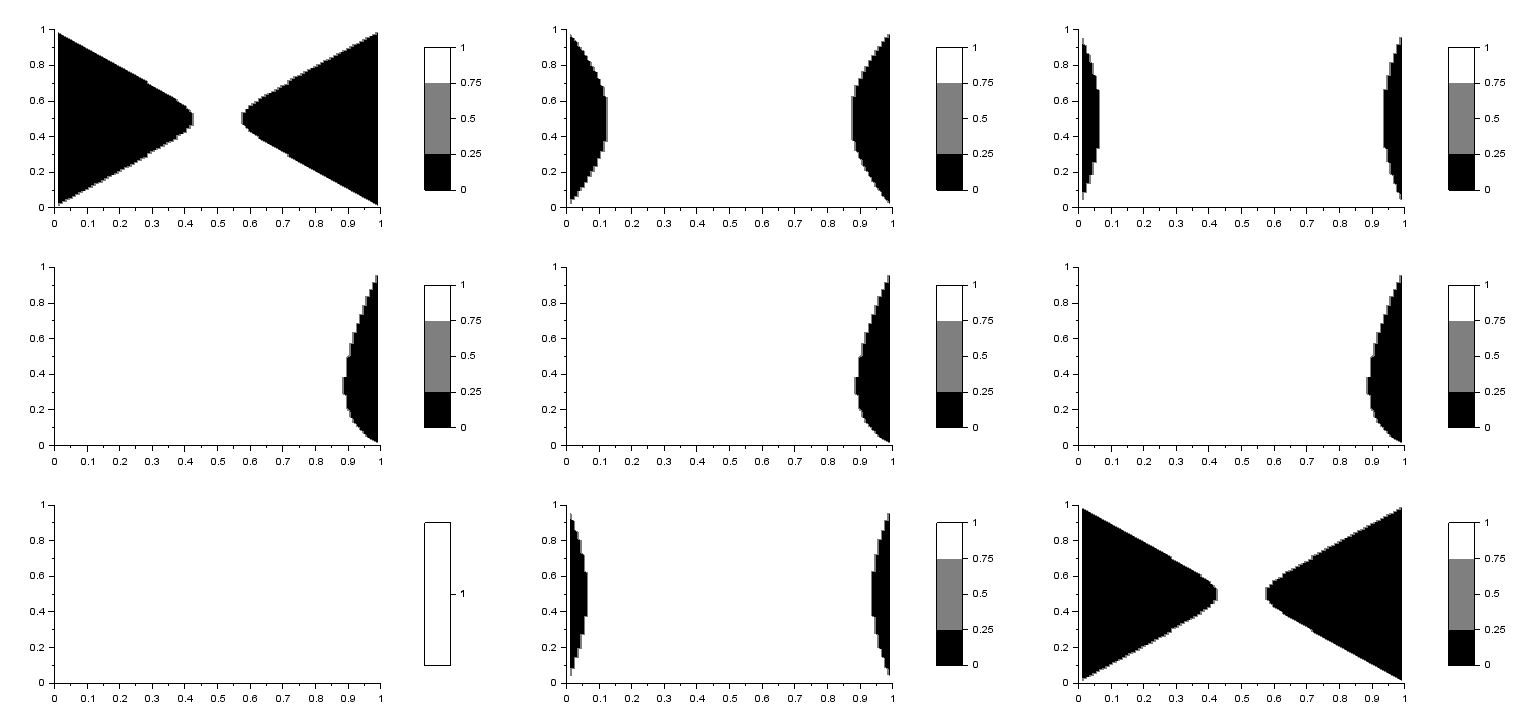}
\caption{Model \eqref{def:toy_model_1_interaction}. In the first row of the figure, regions where $S^1(\mathbb{F})\geqslant   S^2(\mathbb{F})$ (black), $S^1(\mathbb{F})\leqslant   S^2(\mathbb{F})$ (white), and $S^1(\mathbb{F})=   S^2(\mathbb{F})$ (gray) with respect to $p_1$ and $p_2$ varying from 0 to 1 and, from left to right, $p_3=0.01$, $0.5$, and $0.99$. Analogously, the second (resp. last) row considers the regions with $S^1(\mathbb{F})$ and  $S^3(\mathbb{F})$ (resp. $S^2(\mathbb{F})$ and  $S^3(\mathbb{F})$) with respect to $p_1$ and $p_3$ (resp. $p_2$ and $p_3$) varying from 0 to 1 and, from left to right, $p_2=0.01$, $0.5$, and $0.99$ (resp. $p_1=0.01$, $0.5$, and $0.99$).}
\label{fig:toy1_zone_interaction_F}
\end{figure}
\begin{figure}[h!]
\centering 
\includegraphics[width=16cm]{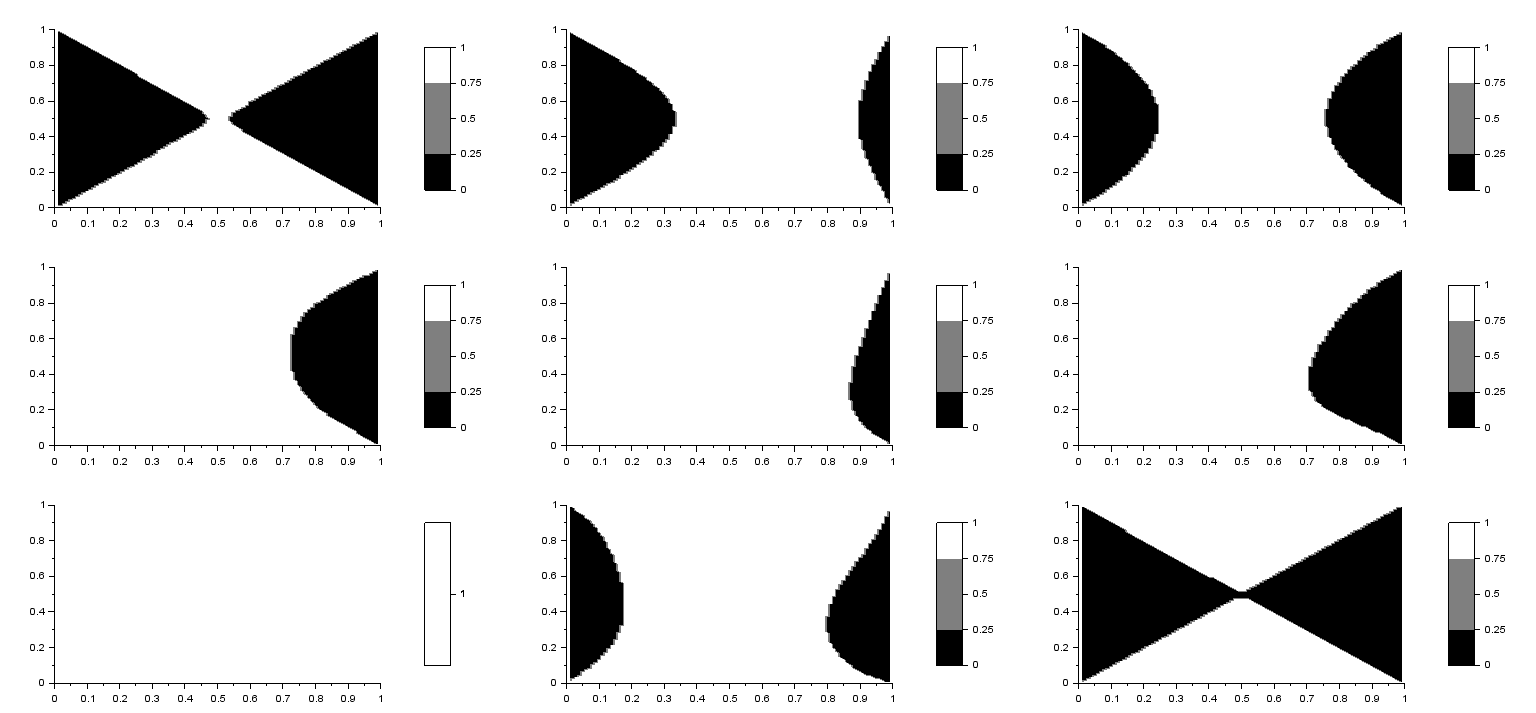}
\caption{Model \eqref{def:toy_model_1_interaction}. In the first row of the figure, regions where $S^1_{2,W_2}\geqslant   S^2_{2,W_2}$ (black), $S^1_{2,W_2}\leqslant   S^2_{2,W_2}$ (white), and $S^1_{2,W_2}=   S^2_{2,W_2}$ (gray) with respect to $p_1$ and $p_2$ varying from 0 to 1 and, from left to right, $p_3=0.01$, $0.5$, and $0.99$. Analogously, the second (resp. last) row considers the regions with $S^1_{2,W_2}$ and  $S^3_{2,W_2}$ (resp. $S^2_{2,W_2}$ and  $S^3_{2,W_2}$) with respect to $p_1$ and $p_3$ (resp. $p_2$ and $p_3$) varying from 0 to 1 and, from left to right, $p_2=0.01$, $0.5$, and $0.99$ (resp. $p_1=0.01$, $0.5$, and $0.99$).}
\label{fig:toy1_zone_interaction_W}
\end{figure}

\medskip

In order to compare the estimation accuracy of the Pick-Freeze method and the rank-based method at a fixed size, we assume that only $N=450$ calls of the computer code  are allowed to estimate the indices $S^{\textbf u}(\mathbb{F})$ and $S^{\textbf{u}}_{2,W_2}$ for $\textbf{u}=\{1\}$, $\{2\}$, and $\{3\}$. We only focus on the first order indices since, as explained previously, the rank-based procedure has not been developed yet for higher order indices. 
We repeat the estimation procedure 500 times. The boxplots of the mean square errors for the estimation of the Fr\'echet indices $S^{\textbf u}(\mathbb{F})$ and the Wasserstein indices $S^{\textbf{u}}_{2,W_2}$ have been plotted in Figure \ref{fig:boxplot_mse_both}. We observe that, for a fixed sample size $N=450$ (corresponding to a Pick-Freeze sample size $N=64$), the rank-based estimation procedure performs much better than the Pick-Freeze method with significantly lower mean errors.

%
%

\begin{figure}
\centering
\begin{tabular}{cc}
\includegraphics[scale=.5]{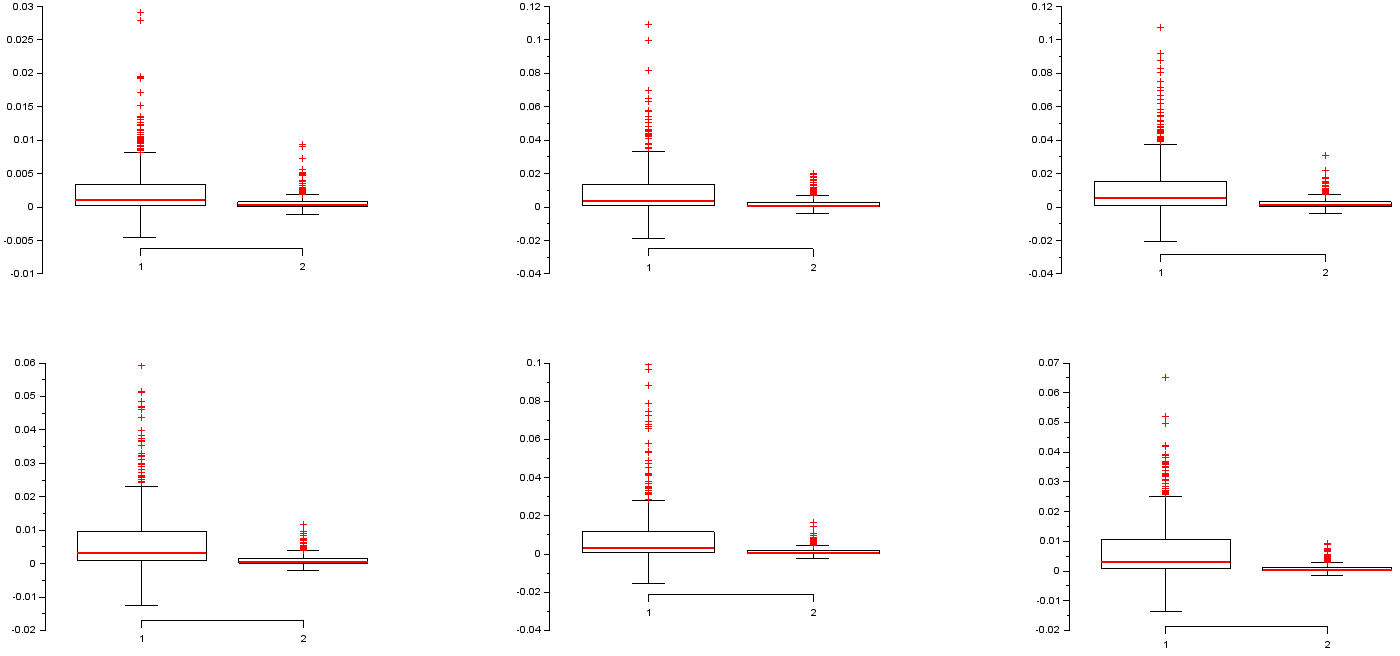}
\end{tabular}
\caption{Model \eqref{def:toy_model_1_interaction} with $p_1=1/3$, $p_2=2/3$, and $p_3=3/4$. Boxplots of the mean square errors of the estimation of the Fr\'echet indices $S^{\textbf u}(\mathbb{F})$ (top row) and the Wasserstein indices $S^{\textbf{u}}_{2,W_2}$ (bottom row) with a fixed sample size and 500 replications. The indices with respect to $\textbf u=\{1\}$, $\{2\}$, and $\{3\}$ are displayed from left to right. The results of the Pick-Freeze estimation procedure with $N=64$ are provided in the left side of each graphic. The results of the rank-based methodology with $N=450$ are provided in the right side of each graphic.  }
\label{fig:boxplot_mse_both}
\end{figure}

\end{ex}

%

\section{Sensitivity analysis for stochastic computer codes}
\label{sec:codesto}

This section deals with stochastic computer codes in the sense that two evaluations of the code for the same input lead to different outputs.

\subsection{State of the art}\label{ssec:art_code_sto}

A first natural way to handle stochastic computer codes is definitely to consider the expectation of the output code. Indeed, as mentioned in \cite{BILGLR16}, previous works dealing with
stochastic simulators together with robust design or optimization and sensitivity analysis consist mainly in approximating the mean and variance of the stochastic output \cite{dellino2015uncertainty,bursztyn2006screening,kleijnen2015design,
ankenman2008stochastic} and then performing a global sensitivity analysis on the expectation of the output code \cite{marrel2012global}. 

As pointed out by \cite{IKL16}, another approach is to consider that the stochastic code is of the form $f(X,D)$ where the random element $X$ contains the classical input variables and the variable $D$ is  an extra unobserved random input.  

Such an idea was exploited in \cite{janon2012asymptotic} to compare the estimation of the Sobol indices in an ``exact'' model to the estimation of the Sobol indices in an associated metamodel.

Analogously, the author of \cite{mazo2019optimal} assumes the existence of an extra random variable $D$ which is not chosen by the practitioner but rather generated at each computation of the output $Y$ independently of $X$. In this framework, he builds two different indices.   
The first index is obtained by substituting $f(X,D)$ for $f(X)$ in the classical definition of the first order Sobol index $S^i=\Var(\E[f(X)|X_i])/\Var(f(X))$. In this case, $D$ is considered as another input, even though it is not observable. The second index is obtained by substituting $\E[f(X,D)|X]$ for $f(X)$ in the Sobol index. The noise is then smoothed out. 

Similarly, the authors of \cite{hart2017efficient} traduces the randomness of the computer code using such an extra random variable $D$. In practice, their algorithm returns $m$ realizations of the first order Sobol indices $S^i$ for $i=1,\ldots,p$, denoted by $\hat S^i_j(d_j)$ for $j=1,\ldots,m$ and $i=1,\ldots,p$. Then, for any $i=1,\ldots,p$,  they approximate the statistical properties of $S^i$ by considering the sample $r$-th moments given by
\begin{align*}
\hat\mu^i_r= \frac 1m \sum_{j=1}^m (\hat S^i_j(d_j))^r
\end{align*}
noticing that
\[
\E_D[\hat\mu^i_r]=\E[(\hat S^i)^r] \quad \text{and} \quad  \Var_D(\hat\mu^i_r)=\frac 1M \Var_D((\hat S^i)^r).
\]

\subsection{The space \texorpdfstring{$\mathcal{W}_q$}{f} as an ideal version of stochastic computer codes}\label{ssec:ideal}


When dealing with stochastic computer codes, the practitioner is generally  interested in the distribution $\mu_x$ of the output for a given $x$. 
As previously seen, one can translate this type of codes in terms of a deterministic code by considering an extra input which is not chosen by the practitioner itself but which is a latent variable generated randomly by the computer code and independently of the classical input. As usual in the framework of sensitivity analysis, one considers the input as a random variable.  All the random variables (the one chosen by the practitioner and the one generated by the computer code) are built on the same probability space, leading to the function $f_s$:
\begin{eqnarray}\label{def:code_stoch}
f_s  : & E \times  \mathcal D & \to  \R\\
& (x,D) & \mapsto  f_s(x,D)\nonumber,
\end{eqnarray}
where $D$ is the extra random variable lying in $\mathcal D$. We naturally denote the output random variable $f_s(x,\cdot{})$ by $f_s(x)$.

Hence, one may define another (deterministic) computer code associated with $f_s$ whose output is the probability measure:
\begin{eqnarray}\label{def:code}
f  : & E  & \to  \mathcal W_q(E)\\
& x & \mapsto  \mu_x.\nonumber
\end{eqnarray}

The framework of \eqref{def:code} is exactly the one of Section \ref{ssec:GSAWS_balls} and has already been handled. Obviously, in practice, one does not assess the output of the code $f$ but 
one can only obtain an empirical approximation of the measure $\mu_x$ given by $n$ evaluations of $f_s$ at $x$, namely, 
\[
\mu_{x,n}= \frac 1n \sum_{k=1}^n \delta_{f_s(x,D_k)}.
\]
Further, \eqref{def:code} can be seen as an ideal version of \eqref{def:code_stoch}. 
Concretely, for a single random input $X\in E=E_1\times\dots \times E_p$, we will evaluate $n$ times the code $f_s$ defined by \eqref{def:code_stoch} (so that the code will generate independently $n$ hidden variables $D_1$, \dots, $D_n$) and one may observe
\[
f_s(X,D_1),\dots , f_s(X,D_n)
\]
leading to the measure
$\mu_{X,n}=  \sum_{k=1}^n \delta_{f_s(X,D_k)}/n$ approximating the distribution of $f_s(X)$. We emphasize on the fact that the random variables $D_1,\dots, D_n$ are not observed.

\subsection{Sensitivity analysis}\label{ssec:SA_code_sto}

Let us now present the methodology we adopt.
In order to study the sensitivity of the distribution $\mu_x$, one can use the framework introduced in Section \ref{ssec:GSAWS_balls} and the index $S_{2,W_q}^{\textbf u}$ given by \eqref{eq:SobolW}.

In an ideal scenario which corresponds to the framework of \eqref{def:code}, one may asses to the probability measure $\mu_x$ for any $x$. Then following the estimation procedure of Section \ref{ssec:GSAWS_est}, one gets an estimation of the sensitivity index $S_{2,W_q}^{\textbf u}$ with good asymptotic properties \cite[Theorem 2.3]{GKLM19}.

In the more realistic framework presented above in \eqref{def:code_stoch}, we only have access to the approximation $\mu_{x,n}$ of $\mu_x$ rendering more complex the  estimation procedure and the study of the asymptotic properties. In this case, the general design of experiments is the following:
\begin{eqnarray*}
&(X_1,D_{1,1},\ldots,D_{1,n}) & \to  ~~~ f_s(X_1,D_{1,1}),\dots , f_s(X_1,D_{1,n}),\\
&(X_1^{\textbf u},D_{1,1}',\dots,D_{1,n}') & \to ~~~ f_s(X_1^{\textbf u},D_{1,1}'),\dots , f_s(X_1^{\textbf u},D_{1,n}'),\\
&& \vdots   \\
&(X_N,D_{N,1},\dots,D_{N,n}) & \to ~~~ f_s(X_N,D_{N,1}),\dots , f_s(X_N,D_{N,n}),\\
&(X_N^{\textbf u},D_{N,1}',\dots,D_{N,n}') & \to ~~~ f_s(X_N^{\textbf u},D_{N,1}'),\dots , f_s(X_N^{\textbf u},D_{N,n}'),
\end{eqnarray*}
where $2\times N\times n$ is the total number of evaluations of the stochastic computer code \eqref{def:code_stoch}. Then we construct the approximations of $\mu_{j}$ (standing for $\mu_{X_j}$) for any $j=1,\dots,N$ given by 
\begin{align}\label{def:measures_approx}
\mu_{j,n} =\frac1n \sum_{k=1}^n \delta_{f_s(X_j,D_{j,k})}.
\end{align}
From there, one may use one of the three estimation procedures presented in Section \ref{ssec:GMSindex}. 

\begin{itemize}
\item \underline{First method - Pick-Freeze}. It suffices to plug the empirical version $\mu_n$ of each measure $\mu$ under concern in \eqref{eq:SobolW_PF}.

\item \underline{Second method - U-statistics}. For $l=1,\dots,4$, let
\begin{align}\label{def:estU_n}
U_{l,N,n}&= \begin{pmatrix}N\\m(l)\end{pmatrix}^{-1}\sum_{1\leq i_1<\dots<i_{m(l)}\leq N}\Phi_l^s\left(
\bmu_{i_1,n},\dots,\bmu_{i_{m(l)},n}
\right)
\end{align}
where as previously seen $\Phi^s_{\cdot{}}$ is the symmetrized version of $\Phi_{\cdot{}}$ defined in \eqref{def:Phi} and $\bmu=(\mu,\mu^{\textbf u})$. Then we estimate $S_{2,W_q}^{\textbf u}$ by 
\begin{equation}\label{def:est_n}
\widehat S_{2,W_q,\text{Ustat},n}^{\textbf u}= \frac{U_{1,N,n}-U_{2,N,n}}{U_{3,N,n}-U_{4,N,n}}.
\end{equation}

\item \underline{Third method - Rank-based}. The rank-based estimation procedure may also easily extend to this context by using the empirical version $\mu_n$ of each measure $\mu$ under concern instead of the true one $\mu$, as explained into more details in the numerical study developed in Section \ref{ssec:art_code_sto}.
\end{itemize}

Actually, these estimators are easy to compute since for two discrete measures supported on a same number of points and given by
\[
\nu_1=\frac 1n \sum_{k=1}^n \delta_{x_k},\; \nu_2=\frac 1n \sum_{k=1}^n \delta_{y_k}, 
\]
the Wasserstein distance between $\nu_1$ and $\nu_2$ simply writes
\begin{align}\label{eq:mes_emp}
W_q^q(\nu_1,\nu_2)=\frac 1n \sum_{k=1}^n (x_{(k)}-y_{(k)})^q,
\end{align}
where $x_{(k)}$ is the $k$-th order statistics of $x$.

\subsection{Central limit theorem for the estimator based on U-statistics}\label{ssec:CLT_code_sto}

\begin{prop}\label{prop:density}
Consider three i.i.d.\ copies $X_1$, $X_2$ and $X_3$ of a random variable\ $X$.
Let $\delta(N)$ be a sequence tending to 0 as $N$ goes to infinity and such that
\[
\P\left(\abs{W_q(\mu_{X_1},\mu_{X_3})- W_q(\mu_{X_1},\mu_{X_2})}\leqslant \delta(N)\right)=o\left(\frac{1}{\sqrt N}\right).
\]
Let $n$ such that $\E[W_q(\mu_{X},\mu_{X,n})]=o(\delta(N)/\sqrt N)$. 
Under the assumptions of \cite[Theorem 2.3]{GKLM19}, we get, for any $\textbf u \subset \{1,\cdots,p\}$,
\begin{align}\label{eq:clt_density}
\sqrt{N}\left(\widehat S_{2,W_q,\text{Ustat},n}^{\textbf u}- S_{2,W_q}^{\textbf u}\right)\xrightarrow[n\to+\infty]{\mathcal{L}}\mathcal{N}(0,\sigma^2)
\end{align}
where the asymptotic variance $\sigma^2$ is given by Equation (13) in the proof of Theorem 2.3 in \cite{GKLM19}.
\end{prop}
%
%
In some particular frameworks, one may derive easily a suitable value of $\delta(N)$. Two examples are given in the following.

\begin{ex}
If the inverse of the random variable $W= \abs{W_q(\mu_{X_1},\mu_{X_3})- W_q(\mu_{X_1},\mu_{X_2})}$ has a finite expectation, then, by Markov inequality,
\begin{align*}
\P\left(W\leqslant \delta(N)\right) = \P\left(W^{-1}\geqslant \delta(N)^{-1}\right)\leqslant 
\frac{1}{\delta(N)} \E\left[\frac{1}{W}\right]
\end{align*}
and it suffices to choose $\delta(N)$ so that $\delta(N)^{-1}=o\left(N^{-1/2}\right)$ as $N$ goes to infinity.
\end{ex}

\begin{ex}[Uniform example]\label{ex:uniform_ex}
Assume that $X$ is uniformly distributed on $[0,1]$ and that $\mu_X$ is a Gaussian distribution centered at $X$ with unit variance. Then the Wasserstein distance $W_2(\mu_{X_1},\mu_{X_2})$ rewrites as $(X_1-X_2)^2$ so that the random variable $W=\abs{W_2(\mu_{X_1},\mu_{X_3})- W_2(\mu_{X_1},\mu_{X_2})}$ is given by
\[
\abs{(X_1-X_3)^2- (X_1-X_2)^2}=\abs{(X_3-X_2)(X_2+X_3-2X_1)}.
\]
Consequently, 
\begin{align*}
\P(W\leqslant \delta(N))&\leqslant \P(\abs{X_3-X_2}\leqslant \sqrt{\delta(N)})+\P(\abs{X_2+X_3-2X_1}\leqslant \sqrt{\delta(N)}).
\end{align*}
Notice that $\abs{X_3-X_2}$ is triangular distributed with parameter $a=0$, $b=1$, and $c=0$ leading to 
\[
\P(\abs{X_3-X_2}\leqslant\alpha)=\alpha(2-\alpha), \quad \text{for all $\alpha\in [0,1]$}.
\]
In addition, 
\begin{align*}
\P(\abs{X_2+X_3-2X_1}\leqslant \sqrt{\delta(N)})&\leqslant \P(\abs{\abs{X_2-X_1}-\abs{X_3-X_1}}\leqslant \sqrt{\delta(N)})\\
&=\int_0^1  \P(\abs{\abs{X_2-u}-\abs{X_3-u}}\leqslant \sqrt{\delta(N)})du.
\end{align*}
Now, $X_2-u$ and $X_3-u$ are two independent random variables uniformly distributed on $[-u,-u]$. Then (see Figure \ref{fig:dessin_jc}), one has 
\[
\P(\abs{\abs{X_2-u}-\abs{X_3-u}}\leqslant \alpha)\leqslant 4\alpha,
\]
whence
\[
\P(\abs{X_2+X_3-2X_1}\leqslant \sqrt{\delta(N)})\leqslant 4\sqrt{\delta(N)}.
\]
Thus it turns out that 
 $\P(W\leqslant \delta(N))=O\Bigl(\sqrt{\delta(N)}\Bigr)$. Consequently, a suitable choice for $\delta(N)$ is $\delta(N)=o(1/N)$.
 
 \begin{center}
  \begin{tikzpicture}[scale=9,
    font=\sffamily,
    level/.style={black,thick},
    transition/.style={black,->,>=stealth',shorten >=1pt},
    radiative/.style={transition,decorate,decoration={snake,amplitude=1.5}},
    nonradiative/.style={transition,dashed},
  ]

	\coordinate (S00) at (-0.25, -0.25); 
	\coordinate (S10) at (0.75, -0.25);
	\coordinate (S20) at (0.75,0.75);
	\coordinate (S30) at (-0.25, 0.75);

	\draw (S00)-- (S10);
    \draw (S10)-- (S20);
  	\draw (S20)-- (S30);
	\draw (S30)-- (S00);
	
	\draw (0.2,0)-- (0.45,-0.25);
    \draw (0,-0.2)-- (0.05,-0.25);
  	\draw (0.2,0)-- (0.75,0.55);
	\draw (0,0.2)-- (0.55,0.75);

	\draw (-0.2,0)-- (-0.25,0.05);
    \draw (0,0.2)-- (-0.25,0.45);
  	\draw (-0.2,0)-- (-0.25,-0.05);
	\draw (0,-0.2)-- (-0.05,-0.25);

    \fill[gray!20] (0.2,0) -- (0.75,0.55) -- (0.75,0.75) -- (0.55,0.75) -- (0,0.2) -- (-0.25,0.45) -- (-0.25,0.05) -- (-0.2,0) -- (-0.25,-0.05) -- (-0.25,-0.25) -- (-0.05, -0.25) -- (0,-0.2) -- (0.05,-0.25) -- (0.45,-0.25) -- (0.2,0);
 
 	\draw [dashed] (S00) -- (S20);
	\draw [dashed] (0.25,-0.25) -- (-0.25,0.25);

    \draw (0.75,0) node[above right]{$u$} ;
    \draw (0,0.75) node[above right]{$u$} ;
    \draw (-0.25,0) node[above left]{$u-1$} ;
    \draw (0,-0.25) node[below right]{$u-1$} ;
    \draw (0.22,0) node[above right]{$\alpha$} ;
    \draw (0,0.23) node[above right]{$\alpha$} ;
    \draw (0,-0.2) node[above right]{$-\alpha$} ;
    \draw (-0.2,0) node[above right]{$-\alpha$} ;

    \draw (0,0) node[below right]{$0$} ;

    \draw (0.75,0)  node{$\bullet$};
    \draw (0,0.75)  node{$\bullet$};
    \draw (-0.25,0)  node{$\bullet$};
    \draw (0,-0.25)  node{$\bullet$};
    
	\draw [->] (-0.35,0) -- (0.85,0);
	\draw [->] (0,-0.35) -- (0,0.85);

  \end{tikzpicture}
  
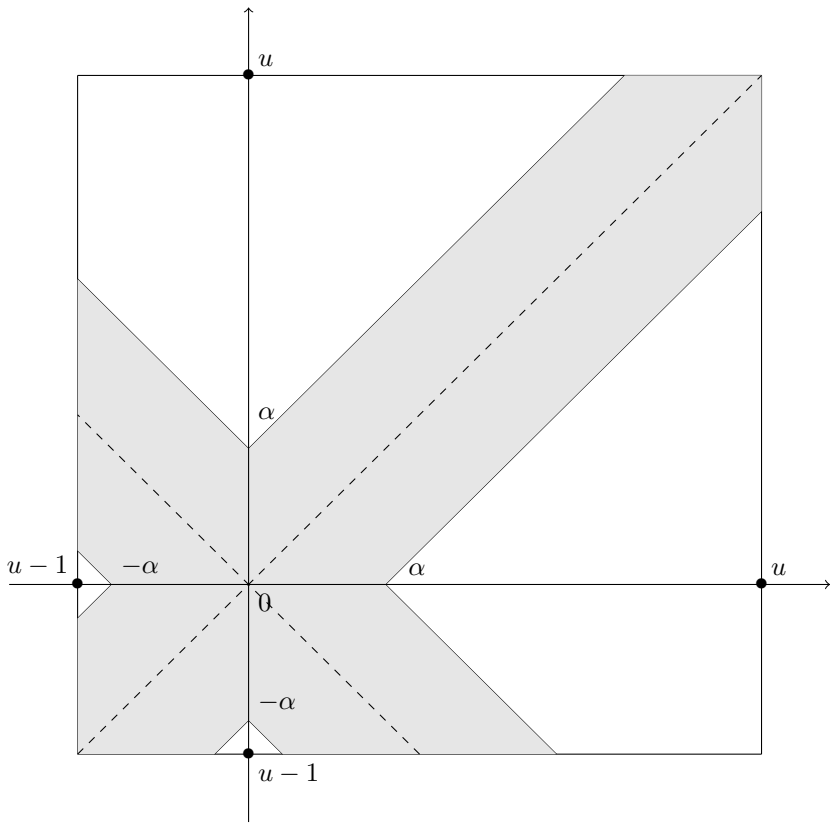
\captionof{figure}{Domain $\Gamma_{u,\alpha}=\{(x_1,x_2)\in [0,1];\; \abs{\abs{x_1-u}-\abs{x_2-u}}\leqslant \alpha\}$ (in grey).}\label{fig:dessin_jc}
\end{center}

\end{ex}


Analogously, one may derive suitable choices for $n$ in some particular cases. For instance, we refer the reader to \cite{BL18} to get upper bounds on $\E[W_q(\mu_X,\mu_{X,n})]$ for several values of $q\geqslant 1$ and several assumptions on the distribution on $\mu_X$: general, uniform, Gaussian, beta, log concave... Here are some results.
\begin{itemize}
\item In the general framework, the upper 
bound for $q\geqslant 1$ relies on the functional 
\[
J_q(\mu_X)= \int_{\R} \frac{\left(F_{\mu_X}(x)(1-F_{\mu_X}(x))\right)^{q/2}}{f_{\mu_X}(x)^{q-1)}}dx
\]
where $F_{\mu_X}$ is the c.d.f.\ associated to $\mu_X$ and $f_{\mu_X}$ its p.d.f. See Cf. \cite[Theorems 3.2, 5.1 and 5.3]{BL18}.
\item Assume that $\mu_X$ is uniformly distributed on $[0,1]$. Then by \cite[Theorems 4.7, 4.8 and 4.9]{BL18},
for any $n\geqslant 1$, 
\[
\E[W_2(\mu_X,\mu_{X,n})^2]\leqslant \frac{1}{6n},
\]
for any $q\geqslant 1$ and for any $n\geqslant 1$, 
\[
\E[W_q(\mu_X,\mu_{X,n})^q]^{1/q}\leqslant (Const) \sqrt{ \frac{q}{n}}.
\]
and  for any $n\geqslant 1$, 
\[
\E[W_{\infty}(\mu_X,\mu_{X,n})]\leqslant \frac{(Const)}{n}.
\]
E.g. $(Const)=\sqrt{\pi/2}$.
\item Assume that $\mu_X$ is a log-concave distribution with standard deviation $\sigma$. Then by \cite[Corollaries 6.10 and 6.12]{BL18}, 
for any $1\leqslant q<2$ and for any $n\geqslant 1$, 
\[
\E[W_q(\mu_X,\mu_{X,n})^q]\leqslant \frac{(Const)}{2-q} \left( \frac{\sigma}{\sqrt{n}}\right)^q,
\]
for any $n\geqslant 1$, 
\[
\E[W_2(\mu_X,\mu_{X,n})^2]\leqslant \frac{(Const)\sigma^2\log n}{n},
\]
and for any $q>2$ and for any $n\geqslant 1$,
\[
\E[W_q(\mu_X,\mu_{X,n})^q]\leqslant \frac{C_q \sigma^q}{n},
\]
where $C_q$ depends on $q$, only.
Furthermore, if $\mu_X$ supported on $[a,b]$, then for any $n\geqslant 1$, 
\[
\E[W_2(\mu_X,\mu_{X,n})^2]\leqslant \frac{(Const)(b-a)^2}{n+1}.
\]
E.g. $(Const)=4/\ln 2$. Cf. \cite[Corollary 6.11]{BL18}.

\end{itemize}

\noindent
\textbf{Example \ref{ex:uniform_ex} - continued}. 
  We consider that $X$ is uniformly distributed on $[0,1]$ and $\mu_X$ is a Gaussian distribution centered at $X$ with unit variance. Then, by \cite[Corollary 6.14]{BL18}, we have
for any $n\geqslant 3$, 
\[
\E[W_2(\mu_X,\mu_{X,n})^2]\leqslant \frac{(Const)\log\log n}{n},
\]
and for any $q>2$ and for any $n\geqslant 3$, 
\[
\E[W_q(\mu_X,\mu_{X,n})^q]\leqslant \frac{C_q}{n(\log n)^{q/2}},
\]
where $C_q$ depends only on $q$. Since we have already chosen $\delta(N)=o(N^{-1})$, it remains to take $n$ so that $\log\log n/n=o(N^{-2})$ to fulfill the condition $\E[W_2(\mu_{X},\mu_{X,n})]=o(\delta(N)/\sqrt N)$.

\subsection{Numerical study}\label{ssec:num_study_sto}

\textbf{Example \ref{ex:toy_model} - continued}. Here, we consider again the code given by \eqref{def:toy_model_1_interaction}.
Having in mind the notation of Section \ref{ssec:ideal}, we consider 
the ideal version of the code:
\begin{eqnarray*}
f  : & E  &\to  \mathcal W_q(E)\\
& (X_1,X_2,X_3) &\mapsto  \mu_{(X_1,X_2,X_3)}\nonumber
\end{eqnarray*} 
where $\mu_{(X_1,X_2,X_3)}$ is the uniform distribution on $[0,1+X_1+X_2+X_1X_3]$, the c.d.f.\ of which is $\mathbb F$ given by \eqref{def:toy_model_1_interaction}
and its stochastic counterpart:
\begin{eqnarray}\label{def:code_stoch_ex}
f_s  : & E \times  D & \to  \R\\
& (X_1,X_2,X_3,D)  & \mapsto  f_s(X_1,X_2,X_3,D)\nonumber
\end{eqnarray}
where $f_s(X_1,X_2,X_3,D)$ is a realization of $\mu_{(X_1,X_2,X_3)}$.

Hence, we no longer assume that one may observe $N$ realizations of $\mathbb F$ associated to the $N$ initial realizations of $(X_1,X_2,X_3)$.  Instead, for any of the $N$ initial realizations of $(X_1,X_2,X_3)$, we assess  $n$ realizations of a uniform random variable on $[0,1+X_1+X_2+X_1X_3]$.

In order to compare the estimation accuracy of the Pick-Freeze method and the rank-based method at a fixed size, we assume that only $N=450$ calls of the computer code $f$ are allowed to estimate the indices $S^{\textbf u}(\mathbb{F})$ and $S^{\textbf{u}}_{2,W_2}$ for $\textbf{u}=\{1\}$, $\{2\}$, and $\{3\}$. We only focus on the first order indices since, as explained previously, the rank-based procedure has not been developed yet for higher order indices.  The empirical c.d.f.\ based on the empirical measures $\mu_{i,n}$ for $i=1,\ldots,n$ in \eqref{def:measures_approx} are constructed with $n=100$ evaluations.  
We repeat the estimation procedure 500 times. The boxplots of the mean square errors for the estimation of the Fr\'echet indices $S^{\textbf u}(\mathbb{F})$ and the Wasserstein indices $S^{\textbf{u}}_{2,W_2}$ have been plotted in Figure \ref{fig:boxplot_mse_both_sto}. We observe that, for a fixed sample size $N=450$ (corresponding to a Pick-Freeze sample size $N=64$), the rank-based estimation procedure performs much better than the Pick-Freeze method with significantly lower mean errors.

\begin{figure}
\centering
\begin{tabular}{c}
\includegraphics[scale=.52]{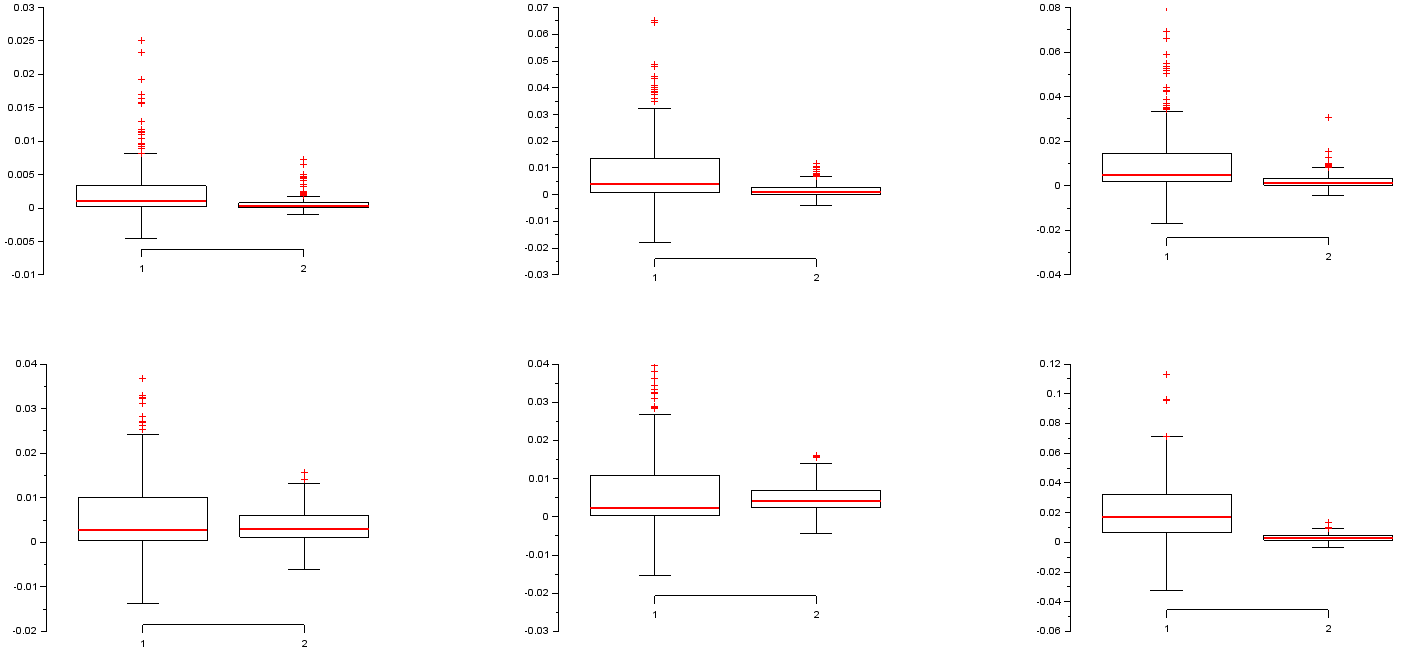}
\end{tabular}
\caption{Model \eqref{def:code_stoch_ex} with $p_1=1/3$, $p_2=2/3$, and $p_3=3/4$. Boxplot of the mean square errors of the estimation of the Fr\'echet indices $S^{\textbf u}(\mathbb{F})$ (top row) and the Wasserstein indices $S^{\textbf{u}}_{2,W_2}$ (bottom row) with a fixed sample size and 200 replications. The indices with respect to $\textbf u=\{1\}$, $\{2\}$, and $\{3\}$ are displayed from left to right. The results of the Pick-Freeze estimation procedure with $N=64$ are provided in the left side of each graphic. The results of the rank-based methodology with $N=450$ are provided in the right side of each graphic.  The approximation size $n$ is fixed at 500.}
\label{fig:boxplot_mse_both_sto}
\end{figure}

\medskip

Another numerical study in the particular setting of stochastic computer codes  and inspired by \cite{GH19marginal} will be considered in Section \ref{ssec:num_app_second_level}.

\section{Sensitivity analysis with respect to the law of the inputs}\label{sec:input}

\subsection{State of the art}\label{ssec:art_input}

The paper \cite{MML19} is devoted to second level uncertainty which corresponds to the uncertainty on the type of the input distributions and/or on the parameters of the input distributions. As mentioned by the authors, such uncertainties can be handled in two different manners: (1) aggregating them with no distinction \cite{chabridon2017evaluation,chabridon2018thesis} or (2) separating them  \cite{MML19}. In \cite{chabridon2017evaluation}, e.g., the uncertainty concerns the parameters of the input distributions. The authors study the expectation with respect to the distribution of the parameters of the conditional output. 
In \cite{chabridon2018thesis}, the second level uncertainties are transformed into first level uncertainties considering the aggregated vector containing the input random variables vector together with the vector of uncertain parameters.
Alternatively, in \cite{MML19}, the uncertainty brought by the lack of knowledge of the input distributions and the uncertainty of the random inputs are treated separately. A double Monte-Carlo algorithm is first considered. 
In the outer loop, a Monte-Carlo sample of input distribution is generated,
while the inner loop proceeds to a global sensitivity analysis associated to each distribution. A more efficient algorithm is also proposed with a unique Monte-Carlo loop. The sensitivity analysis is then performed using the so-called Hilbert-Schmidt dependence
measures (HSIC indices) on  the input distributions rather than the input random variables themselves. See, e.g., \cite{gretton2005measuring} for the definition of the HSIC indices and more details on the algorithms.

In \cite{morio2011influence}, a different approach is adopted. A failure  probability is studied while the uncertainty concerns the parameters of the input distributions. An algorithm with low computational cost is proposed  to handle such uncertainty together with the rare event setting. A single initial sample allows to compute the failure probabilities associated to different parameters of the input distributions. A similar idea is exploited in \cite{lemaitre2015density} in which the authors consider input perturbations. 
and Perturbed-Law based Indices that are used to quantify
the impact of a perturbation of an input p.d.f. on a failure probability. 
Analogously, the authors of \cite{GH19distribution,GH19marginal} are interested in (marginal) p.d.f. perturbations and the aim 
is to study the ``robustness of the Sobol indices to distributional uncertainty
and to marginal distribution uncertainty'' which correspond to  second level uncertainty. For instance, the basic idea of the approach  proposed in \cite{GH19distribution} is to view the total Sobol index as an operator which inputs the
p.d.f. and returns the Sobol index. Then the analysis of robustness is done computing and studying the Fr\'echet derivative of this operator. The same principle is used in \cite{GH19marginal} to treat the robustness with respect to the marginal distribution uncertainty.

Last but not least, it is worth mentioning the classical approach
of epistemic global sensitivity analysis of Dempster-Shafer theory (see, e.g.,  \cite{smets1994dempster,alvarez2009reduction}).
This theory describes the random variables together with an epistemic uncertainty traduced in terms of an associated epistemic variable $Z$ on a set $A$, a mass function representing
a probability measure on the set $\P(A)$ of all subsets $A$. This lack of knowledge leads to an upper and lower bound of the c.d.f.\ and can be viewed as a second level uncertainty.

\subsection{Link with stochastic computer codes}\label{ssec:link_code_sto}

We propose a new procedure that stems from the the methodology in the context of stochastic computer codes described in Section \ref{sec:codesto}.
We still denote by $\mu_i$ ($i=1,\ldots,p$) the distribution of the input $X_i$ ($i=1,\ldots,p$) in the model given by  \eqref{eq:model}. There are several ways to model the uncertainty with respect to the choice of each $\mu_i$. Here we adopt the following framework. We assume that each $\mu_i$ belongs to some family $\mathcal{P}_i$ of probability measures endowed with the probability measure $\mathbf{\mathbb{P}_{\mu_i}}$. In general, there might be measurability issues and the question of how to define a $\sigma-$field on some general spaces  $\mathcal{P}_i$ can be tricky. We will restrict our study to the simple case where the existence of the probability measure  $\mathbf{\mathbb{P}_{\mu_i}}$ on $\mathcal{P}_i$ is given by the construction of the set  $\mathcal{P}_i$. More precisely, we proceed as follows.

\begin{itemize}
\item First, for $1\leqslant i\leqslant p$, let $d_i$ be an integer and let $\Theta_i\subset \R^ {d_i}$. Then consider the probability  space $\left(\Theta_i,\mathcal{B}(\Theta_i),\nu_{\Theta_i}\right)$ where $\mathcal{B}(\Theta_i)$ is the Borel $\sigma-$field and $\nu_{\Theta_i}$ is a probability measure on 
$\left(\Theta_i,\mathcal{B}(\Theta_i)\right)$. 
\item Second, for $1\leqslant i\leqslant p$, we consider an identifiable parametric set of probability measure $\mathcal{P}_i$ on $E_i$:  $\mathcal{P}_i:=\{\mu_{\theta}, \theta\in\Theta_i\}$. Let us denote by $\pi_i$ the one-to-one mapping from $\Theta_i$ to $\mathcal{P}_i$ defined by $\pi_i(\theta):=\mu_\theta \in \mathcal P_i$ and  define the $\sigma-$field $\mathcal{F}_i$ on $\mathcal{P}_i$ by 
\begin{equation*}
A\in\mathcal{F}_i \iff \exists B\in\mathcal{B}(\Theta_i),\ A=\pi_i(B).
\end{equation*}
Then we endow this measurable space with  the   probability $\Pi_i$ defined, for any $A\in\mathcal{F}_i $, by 
\begin{equation*}
\Pi_i(A)=\nu_{\Theta_i}\left(\pi_i^{-1}(A)\right).
\end{equation*}
\item Third, in order to perform a second level sensitivity analysis on   \eqref{eq:model}, we introduce the stochastic  mapping $f_s$ from $\mathcal{P}_1\times\ldots\times\mathcal{P}_p$  to  $\mathcal{X}$ defined by 
\begin{equation}\label{eq:fS}
f_s\left(\mu_1,\ldots,\mu_p\right)= f(X_1,\ldots,X_p)
\end{equation}
where $X_1,\ldots,X_p$ are independently drawn according to the distribution $\mu_1\times \ldots \times \mu_p$. Hence $f_s$ is a stochastic computer code from 
$\mathcal{P}_1\times\ldots\times \mathcal{P}_p$  to  $\mathcal{X}$ and once the probability measures $\P_{\mu_i}$ on each $\mathcal{P}_i$ are defined, we can perform sensitivity analysis using the framework of Section \ref{sec:codesto}.
\end{itemize}

\subsection{Numerical study}\label{ssec:num_app_second_level}


As in \cite{GH19marginal}, let us consider the synthetic example defined on $[0,1]^3$ by
\begin{align}\label{def:gremaud}
f(X_1,X_2,X_3)=2X_2e^{-2X_1}+X_3^2.
\end{align}
We are interested in the uncertainty in the support of the random variables $X_1$, $X_2$ and $X_3$. To do so, we follow the notation and framework of \cite{GH19marginal}. For $i=1$, 2, and 3, we assume that $X_i$  is uniformly distributed on the interval $[A_i,B_i]$, where $A_i$ and $B_i$ are themselves uniformly distributed on $[0,0.1]$ and $[0.9,1]$ respectively. As remarked in \cite{GH19marginal}, it seems natural that $f$ will vary more in the $X_2$-direction when $X_1$ is close to 0 and less when $X_1$ is close to 1.

As mentioned in Section \ref{ssec:art_input}, the authors of \cite{GH19marginal} view the total Sobol index as an operator which inputs the p.d.f.\ and returns the total Sobol index. Then they study the Fr\'echet derivative of this operator and determine the most influential p.d.f., which depends on a parameter denoted by $\delta$. Finally, they make vary this parameter $\delta$.

Here, we adopt the methodology explained in the previous section (Section \ref{ssec:link_code_sto}). Namely, we consider the stochastic computer code given by:
\begin{align}
f_s(\mu_1,\mu_2,\mu_3)=2X_2e^{-2X_1}+X_3^2,
\end{align}
where the $X_i$'s are independently drawn according to the uniform measure $\mu_i$ on $[A_i,B_i]$ with $A_i$ and $B_i$ themselves uniformly distributed on $[0,0.1]$ and $[0.9,1]$ respectively. Then to  estimate the indices $S_{2,W_2}^{i}$, for $i=1$, 2, and 3, we proceed as follows.
\begin{enumerate}
\item For $i=1$, 2, and 3, we produce a $N$-sample $\left([A_{i,j},B_{i,j}]\right)_{j=1,\ldots,N}$ of intervals $[A_i,B_i]$.  
\item For $i=1$, 2, and 3,  and, for $j=1,\ldots,N$, we generate a $n$-sample $\left(X_{i,j,k}\right)_{k=1,\ldots,n}$ of $X_i$, where $X_{i,j,k}$ is uniformly distributed on $[A_{i,j},B_{i,j}]$. 
\item For $j=1,\ldots,N$, we compute the $n$-sample $\left(Y_{j,k}\right)_{k=1,\ldots,n}$ of the output using
\[
Y=f(X_1,X_2,X_3)=2X_2e^{-2X_1}+X_3^2.
\]
Thus we get a $N$-sample of the empirical measures of the distribution of the output $Y$ given by:
\[
\mu_{j,n}=\frac 1n \sum_{k=1}^n \delta_{Y_{j,k}}, \quad \text{for $j=1,\ldots,N$}. 
\]
\item For $i=1$, 2, and 3, we order the intervals $\left([A_{i,j},B_{i,j}]\right)_{j=1,\ldots,N}$ and get the Pick-Freeze versions of $Y$ to treat the sensitivity analysis regarding the $i$-th input. 
\item Finally, it remains to compute the indicators of the empirical version of \eqref{eq:SobolW_PF} using \eqref{eq:mes_emp} and their means to get the Pick-Freeze estimators of $S^{\textbf{u}}_{2,W_2}$, for $\textbf{u}=\{1\}$, $\{2\}$, $\{3\}$, $\{1,2\}$, $\{1,3\}$, and $\{2,3\}$. 
\end{enumerate} 
Notice that we only consider the estimators based on the Pick-Freeze method since we allow for both bounds of the interval to vary
and, as explained previously, the rank-based procedure has not been developed yet neither for higher order indices nor in higher dimensions.

First, we compute the estimators of $S^{\textbf u}_{2,W_2}$ following the previous procedure with a sample size $N=500$ and an approximation size $n=500$. We also perform another batch of simulations allowing for higher variability on the bounds: $A_i$ is now uniformly distributed on $[0,0.45]$ while $B_i$ is now uniformly distributed on $[0.55,1]$. The results are displayed in Table \ref{tab:first}.
\begin{table}[h]\begin{center}
\begin{tabular}{c|c|c|c|c|c|c|c}
 & $\textbf u $ & $\{1\}$ & $\{2\}$ & $\{3\}$ & $\{1,2\}$ & $\{1,3\}$ & $\{2,3\}$ \\
\hline
$A_i\in [0,0.1]$  &&&&&&&\\
$B_i\in [0.9,1]$ & $\hat S^{\textbf u}_{2,W_2}$ & 0.07022 &  0.08791 &  0.09236  & 0.14467  & 0.21839  & 0.19066\\
\hline
$A_i\in [0,0.45]$ &&&&&&&\\
$B_i\in [0.55,1]$ & $\hat S^{\textbf u}_{2,W_2}$ &
   0.11587 &  0.06542  & 0.169529 &  0.22647  & 0.40848 &  0.34913\\
\end{tabular}
\end{center}
\caption{\label{tab:first} Model \eqref{def:gremaud}. GSA on the parameters of the input distributions. Estimations of $S^{\textbf u}_{2,W_2}$ with a sample size $N=500$ and an approximation size $n=500$. In the first row, $A_i$ is uniformly distributed on $[0,0.1]$ while $B_i$ is uniformly distributed on $[0.55,1]$. In the second row, we allow for more variability: $A_i$ is  uniformly distributed on $[0,0.45]$ while $B_i$ is uniformly distributed on $[0.55,1]$.}
\end{table}
Second, we run another simulations allowing for more variability on the upper bound related to the third input $X_3$ only: $B_3$ is uniformly distributed on $[0.5,1]$ (instead of $[0.9,1]$).  The results are displayed in Table \ref{tab:second}. We still use a sample size $N=500$ and an approximation size $n=500$.
\begin{table}[h]
\begin{center}
\begin{tabular}{c|c|c|c|c|c|c}
$\textbf u $ & $\{1\}$ & $\{2\}$ & $\{3\}$ & $\{1,2\}$ & $\{1,3\}$ & $\{2,3\}$ \\
\hline
&&&&&&\\
$\hat S^{\textbf u}_{2,W_2}$ & 0.01196  & 0.06069 &  0.56176 & -0.01723 &  0.63830  & 0.59434\\
\end{tabular}
\end{center}
\caption{\label{tab:second} Model \eqref{def:gremaud}. GSA on the parameters of the input distributions. Estimations of $S^{\textbf u}_{2,W_2}$ with a sample size $N=500$ and an approximation size $n=500$ and more variability on $B_3$, now uniformly distributed on $[0.5,1]$. }
\end{table}
Third, we perform a classical GSA on the inputs rather than on the parameters of their distributions. Namely, we estimate the index $S^{\textbf u}_{2,CVM}$ with a sample size $N=10^4$. The reader is referred to \cite[Section 3]{GKL18} for the definition of the index $S^{\textbf u}_{2,CVM}$ and its Pick-Freeze estimator together with their properties. The results are displayed in Table \ref{tab:third}.
\begin{table}[h]
\begin{center}
\begin{tabular}{c|c|c|c|c|c|c}
$\textbf u $ & $\{1\}$ & $\{2\}$ & $\{3\}$ & $\{1,2\}$ & $\{1,3\}$ & $\{2,3\}$ \\
\hline
&&&&&&\\
$\hat S^{\textbf u}_{2,W_2}$ & 0.13717  & 0.15317  & 0.33889  & 0.33405  & 0.468163  & 0.53536\\
\end{tabular}
\end{center}
\caption{\label{tab:third} Model \eqref{def:gremaud}. Direct GSA on the inputs. Estimations of $S^{\textbf u}_{2,CVM}$  with a sample size $N=10^4$. The reader is referred to \cite[Section 3]{GKL18} for the definition of the index $S^{\textbf u}_{2,CVM}$ and its Pick-Freeze estimator together with their properties.}
\end{table}
%
%
%
%
%

In Table \ref{tab:third}, we see that the Cram\'er-von-Mises index related to $X_3$ is more than twice as important as $X_1$ and $X_2$ (when considering only first order effects).  Nevertheless, when one is interested in the choice of the input distributions of $X_1$, $X_2$, and $X_3$, the first row in Table \ref{tab:first} shows that each choice is equally  important. Now, if one give more freedom to the space where the distribution lives, the relative importance may change as one can see in Table \ref{tab:second} (second row) and in Table \ref{tab:third}.  More precisely, in Table \ref{tab:second}, the variability of the third input distribution (namely, the variability of its upper bound) is five times bigger than the other variabilities. Not surprisingly, it results that the importance of the choice of the third input distribution is then much more important than the choices of the distributions of the two first  inputs.


\section{Conclusion}\label{sec:concl}

In this article, we present a very general way to perform sensitivity analysis when the output $Z$ of a computer code lives in a metric space. The main idea is to consider real-valued squared integrable  test functions $(T_a(Z))_{a\in\Omega}$ parameterized  by a finite number of elements of a probability space. Then  Hoeffding decomposition of the test functions $T_a (Z)$ is computed and integrated with respect to the parameter  $a$.  This very general and flexible definition allows, in one hand, to recover a lot of classical indices (namely, the Sobol indices and the Cram\'er-von-Mises indices)  and, in the other hand, to perform a well tailored  and interpretable sensitivity analysis. Furthermore, a sensitivity analysis is also made possible for computer codes the output of which is a c.d.f., for stochastic computer codes (that are seen as an approximation of c.d.f.-valued computer codes). Last but not least, it enables also to perform second level sensitivity analysis by embedding second level sensitivity analysis as a particular case of stochastic computer codes. 

\paragraph{Acknowledgment}

\appendix

\section{Proofs}\label{sec:proof}

\subsection{Notation}

It is convenient to have short expressions for terms that converge in probability to zero. We follow \cite{van2000asymptotic}.
The notation $o_{\P}(1)$ (respectively $O_{\P}(1)$) stands for a sequence of random variables that converges to zero in probability (resp. is bounded in probability) as $n \to \infty$. More generally, for a sequence of random variables $R_n$,
\begin{align*}
X_n&=o_{\P}(R_n) \quad \textrm{means} \quad X_n=Y_nR_n \quad \textrm{with} \quad Y_n\overset{\P}{\rightarrow}0\\
X_n&=O_{\P}(R_n) \quad \textrm{means} \quad X_n=Y_nR_n \quad \textrm{with} \quad Y_n=O_{\P}(1).
\end{align*}
For deterministic sequences $X_n$ and $R_n$, the stochastic notation reduce to the usual $o$ and $O$.
Finally, $c$ stands for a generic constant that may differ from one line to another.

\subsection{Proof of Proposition \ref{prop:density} }

One has
\[
\sqrt{N} \left(\widehat S_{2,W_q,\text{Ustat},n}^{\textbf u}-S_{2,GMS}^{\textbf u}\right)
=\sqrt{N} \left(\widehat S_{2,W_q,\text{Ustat},n}^{\textbf u}-\widehat{S}_{2,GMS,\text{Ustat}}^{\textbf u}\right)+\sqrt{N} \left(\widehat{S}_{2,GMS,\text{Ustat}}^{\textbf u}-S_{2,GMS}^{\textbf u}\right).
\]
By \cite[Theorem 2.3]{GKLM19}, the second term in the right hand side of the previous equation is asymptotically Gaussian. If we prove that the first term in the right hand side is $o_{\P}(1)$, then by Slutsky's Lemma \cite[Lemma 2.8]{van2000asymptotic}, $\sqrt{N} \left(\widehat{S}_{2,GMS,\text{Ustat},n}^{\textbf u}-S_{2,GMS}^{\textbf u}\right)$ is asymptotically Gaussian.

Now we prove that $\sqrt{N} \left(\widehat{S}_{2,GMS,\text{Ustat},n}^{\textbf u}-\widehat{S}_{2,GMS,\text{Ustat}}^{\textbf u}\right)=o_{\P}(1)$. We write 
\begin{align*}
&\widehat S_{2,W_q,\text{Ustat},n}^{\textbf u}-\widehat{S}_{2,GMS,\text{Ustat}}^{\textbf u} = \Psi(U_{1,N,n},U_{2,N,n},U_{3,N,n},U_{4,N,n})-\Psi(U_{1,N},U_{2,N},U_{3,N},U_{4,N})\\
&=\frac{\left[(U_{1,N,n}-U_{1,N})-(U_{2,N,n}-U_{2,N})\right](U_{3,N}-U_{4,N})}{\left[(U_{3,N,n}-U_{3,N})-(U_{4,N,n}-U_{4,N})+(U_{3,N}-U_{4,N})\right](U_{3,N}-U_{4,N})}\\
&\quad -\frac{\left[(U_{3,N,n}-U_{3,N})-(U_{4,N,n}-U_{4,N})\right](U_{1,N}-U_{2,N})}{\left[(U_{3,N,n}-U_{3,N})-(U_{4,N,n}-U_{4,N})+(U_{3,N}-U_{4,N})\right](U_{3,N}-U_{4,N})}.
\end{align*}
Since $(U_{l,N,n}-U_{l,N,n})$, for $l=3$ and $4$ and $(U_{3,N}-U_{4,N})$ converges almost surely respectively to 0 and $I(\Phi_3)-I(\Phi_4)$, the denominator converges almost surely. Thus it suffices to prove that the numerator is $o_{\P}(1/\sqrt N)$ which reduces to prove that 
 $\sqrt{N}\left(U_{l,N,n}-U_{l,N}\right)=o_{\P}(1)$ for $l=1,\dots,4$, where $U_{l,N,n}$ (respectively $U_{l,N}$) has been defined in \eqref{def:estU_n} (resp. \eqref{def:estU}). Let $l=1$ for example. The other terms can be treated analogously. Here, $m(1)=3$. We write
\begin{align*}
\E&\left[\abs{U_{1,N,n}-U_{1,N}}\right]\\
&\leqslant \begin{pmatrix}N\\3\end{pmatrix}^{-1}(3!)^{-1}\sum_{\substack{1\leq i_1<i_2<i_{3}\leq N\\\tau\in \mathcal S_3}}
\E\left[\abs{\Phi_1\left(\bmu_{\tau(i_1),n},\bmu_{\tau(i_2),n},\bmu_{\tau(i_{3}),n}\right)
-\Phi_1\left(\bmu_{\tau(i_1)},\bmu_{\tau(i_2)},\bmu_{\tau(i_{3})}\right)}\right]\\
&=\E\left[\abs{\Phi_1\left(\bmu_{1,n},\dots\bmu_{2,n},\bmu_{3,n}\right)
-\Phi_1\left(\bmu_1,\bmu_2,\bmu_3\right)} \right]\\
&  \leqslant 2 \E\left[ \abs{\ind_{W_q(\mu_1,\mu_3)\leqslant W_q(\mu_1,\mu_2)}-\ind_{W_q(\mu_{1,n},\mu_{3,n})\leqslant W_q(\mu_{1,n},\mu_{2,n})}}\right]\\
&\eqdef 2 \E\left[B_n\right]
\end{align*}
where the random variable $B_n$ in the expectation in the right hand side of the previous inequality is a Bernoulli random variable whose distribution does not depend on $(1,2,3)$. Let $\Delta(N)$ be the following event 
\[
\Delta(N)=\left\{
\abs{W_q(\mu_{\tau(1)},\mu_{\tau(3)})- W_q(\mu_{\tau(1)},\mu_{\tau(2)})}\geqslant \delta(N)
\right\}.
\]
Obviously, we get
$\E\left[B_n\ind_{\Delta(N)^c}\right]\leqslant \P(\Delta(N)^c)$,
where $A^c$ stands for the complementary of $A$ in $\Omega$. Furthermore,
\begin{align*}
\E\left[B_n\ind_{\Delta(N)}\right]&\leqslant \E\left[B_n\vert\Delta(N)\right]
= \P\left(B_n=1\vert\Delta(N)\right)\\
&\leqslant \sum_{r=1}^3 \P\left(W_q(\mu_{r},\mu_{r,n})\geqslant \frac{\delta(N)}{4}\right)\\
&\leqslant \frac{12}{\delta(N)}\E[W_q(\mu_{1},\mu_{1,n})].
\end{align*}
Finally, we introduce $\varepsilon >0$ and we study:
\begin{align*}
\P\left(\sqrt N\abs{U_{1,N,n}-U_{1,N}}\geqslant \varepsilon \right)
&\leqslant \frac{\sqrt N}{\varepsilon} \E\left[\abs{U_{1,N,n}-U_{1,N}}\right]\\
&\leqslant 2 \frac{\sqrt N}{\varepsilon} \E\left[B_n\right]\\
&\leqslant \frac{\sqrt N}{\varepsilon} \frac{24}{\delta(N)}\E[W_q(\mu_{1},\mu_{1,n})]+ 2\frac{\sqrt N}{\varepsilon} \P(\Delta(N)^c).
\end{align*}
It remains to choose first, $\delta(N)$ so that $\P(\Delta(N)^c)=o\left(1/\sqrt N\right)$
and second, $n$ such that $\E[W_q(\mu_{1},\mu_{1,n})]=o(\delta(N)/\sqrt N )$. Consequently, $\sqrt N(U_{1,N,n}-U_{1,N})=o_{\P}(1)$. Analogously, one gets $\sqrt N(U_{l,N,n}-U_{l,N})=o_{\P}(1)$ for $l$=2, 3 and 4.

\bibliographystyle{abbrv}
\bibliography{biblio_Wasserstein_AS}

\end{document}